%% file: blocking-pg-v1-final.tex
\documentclass{article} 
\usepackage[margin=3.5cm]{geometry}
\RequirePackage{natbib}

\usepackage{algorithm,algorithmic}

\usepackage{amsthm,amsmath,amssymb}
\usepackage{aliascnt}
\usepackage{hyperref}
\usepackage{enumitem}
\usepackage{ float, xargs}
\usepackage{cases}

\newtheorem{theorem}{Theorem}
\newaliascnt{proposition}{theorem}
\newtheorem{proposition}[proposition]{Proposition}
\aliascntresetthe{proposition}

\newaliascnt{lemma}{theorem}
\newtheorem{lemma}[lemma]{Lemma}
\aliascntresetthe{lemma}

\newaliascnt{corollary}{theorem}
\newtheorem{corollary}[corollary]{Corollary}
\aliascntresetthe{corollary}

\newaliascnt{definition}{theorem}

\aliascntresetthe{definition}

\newaliascnt{example}{theorem}

\aliascntresetthe{example}

\newaliascnt{remark}{theorem}
\newtheorem{remark}[remark]{Remark}
\aliascntresetthe{remark}

\usepackage{graphicx,tikz}
\usetikzlibrary{arrows,backgrounds}

\usetikzlibrary{matrix}
\graphicspath{{./figs/}}

\usepackage[textwidth=3cm, textsize=footnotesize]{todonotes}
\usepackage{color}

\usepackage{xspace}
\newcommand\jsd{JSD\@\xspace}
\newcommand\pg{PG\@\xspace}
\newcommand\smc{SMC\@\xspace}

\newcommand\hmm{HMM\@\xspace}
\newcommand\mcmc{MCMC\@\xspace}
\newcommand\pmcmc{PMCMC\@\xspace}
\newcommand{\wrt}{with respect to\@\xspace}

\newcommand{\ie}{i.e.\xspace}
\newcommand{\eg}{e.g.\xspace}
\newcommand{\cf}{cf.\xspace}
\newcommand{\pdf}{probability density function\xspace}

\newcommand\join[2]{\{#1,#2\}}

\DeclareMathOperator\osc{osc}
\renewcommand\mid{\,\vert\,}
\newcommand{\+}{\mathsf{T}}                      
\newcommand{\range}[2]{#1, \, \dots, \, #2}      
\newcommand{\prange}[2]{(#1, \, \dots, \, #2)}   
\newcommand{\crange}[2]{\{#1, \, \dots, \, #2\}}
\newcommand{\rmd}{\mathrm{d}}
\newcommand{\eqsp}{\;}
\newcommand{\eqdef}{\ensuremath{:=}}

\newcommandx\sequence[3][2=t,3=\zset]
{\ifthenelse{\equal{#3}{}}{\ensuremath{\{ #1_{#2}\}}}{\ensuremath{\{ #1_{#2}, \eqsp #2 \in #3 \}}}}
\newcommand\card[1]{|#1|} 
\newcommand\I{\mathbb{I}}  

\newcommand\Prb{\mathbb{P}}

\newcommand\h{h} 

\newcommand\rset{\ensuremath{\mathbb{R}}}
\newcommand\nset{\ensuremath{\mathbb{N}}}

\newenvironment{hyp}[1]{
\begin{enumerate}[label=(\textbf{\sf #1}-\arabic*),resume=hyp#1]\begin{sf}}
{\end{sf}\end{enumerate}}

\newcommand\hypref[1]{{\textnormal\textbf{\sf\ref{#1}}}}

\newcommand\T{n}
\newcommand\Np{N}
\newcommand\Cpl{\Psi} 

\newcommand\wf{w}
\newcommand{\ewght}[2]{\omega_{#1}^{#2}}
\newcommand{\epart}[2]{X_{#1}^{#2}}
\newcommandx{\ewghtfunc}[4][1=]{w^{#1}(#3, #4)}
\newcommandx{\ewghtfuncfin}[5][1=]{\check w^{#1}(#3, #4, #5)}
\newcommandx{\kis}[4][1=]{r^{#1}(#3, #4)}
\newcommandx{\kisfin}[5][1=]{\check r^{#1}(#3, #5, #4)}
\newcommand{\Xset}{\ensuremath{\mathsf{X}}}

\newcommand{\Yset}{\ensuremath{\mathsf{Y}}}

\newcommand{\Xinitv}{\mu}
\newcommand\XXset{\Xset^{\T}}

\newcommand\J{\mathcal{J}}
\newcommand\W{\mathcal{W}}
\newcommand\one{\mathbf{1}}
\newcommand\bp{\partial}
\newcommand\si{s}
\newcommand\ui{u}
\newcommand\kernel[3]{#1(#2,#3)}

\newcommand\extended[1]{#1^+} 

\newcommand\compP{\mathcal{P}}
\newcommand\Y{y}

\newcommand\rate{\lambda}
\newcommand\ratetwo{\lambda'}
\newcommand\maxsumW{\beta}
\newcommand\maxW{\gamma}
\newcommand\fixL{L}  
\newcommand\p{p} 

\title{Blocking Strategies and Stability of Particle Gibbs Samplers}
\date{September 28, 2015}

\author{
  Sumeetpal S.\ Singh
  \and
  Fredrik Lindsten \and Eric Moulines
}

\begin{document}
\maketitle

\begin{abstract}
Sampling from the conditional (or posterior) probability distribution of the latent states of a Hidden Markov Model, given the realization of the observed process,
is a non-trivial problem in the context of Markov Chain Monte Carlo.
To do this \citet{AndrieuDH:2010} constructed a Markov kernel which leaves this conditional distribution invariant using a Particle Filter.
From a practitioner's point of view, this Markov kernel attempts to mimic the act of sampling all the latent state variables as one block from the posterior distribution but for models where exact simulation is not possible.
There are some recent theoretical results that establish the uniform ergodicity of this Markov kernel and that the mixing rate does not diminish provided
the number of particles grows at least linearly with the number of latent states in the posterior.
This gives rise to a cost, per application of the kernel, that is quadratic in the number of latent states which
could be prohibitive for long observation sequences. We seek to answer an obvious but important question: is there a different implementation with a cost per-iteration that grows
linearly with the number of latent states, but which is still stable in the sense that its mixing rate does not deteriorate? We address this problem using blocking strategies,  which are easily
parallelizable, and prove stability of the resulting sampler.
\end{abstract}



\section{Introduction}
Let $\{(X_t,Y_t) \in (\Xset,\Yset), t \in \nset_+\}$ be a Markov chain evolving according to a (general state space) hidden Markov model (\hmm).
That is, the \emph{state process} $\sequence{X}[t][\nset_+]$ is
a Markov chain with \emph{state space} $\Xset$, evolving according
to a Markov transition kernel $M$ and with initial distribution $\Xinitv$ (\ie, $\Xinitv$ is the distribution of the state at time $t=1$).
The sequence $\sequence{X}[t][\nset_+]$ is not directly observed and inference needs to be carried out based on the
observations $\sequence{Y}[t][\nset_+]$ only. Conditionally on $\sequence{X}[t][\nset_+]$, the observations $\sequence{Y}[t][\nset_+]$
are independent. The conditional density of $Y_t$ given $X_t = x_t$, with respect to
some dominating measure is denoted by $g(x_t, y_t)$.

We will work under the assumption that a \emph{fixed} sequence of observations
$y_{1:\T} \eqdef \prange{y_1}{y_\T}$ is available, where $\T$ is some final time point. The key object of interest
is then the \emph{joint smoothing distribution}  (\jsd) $\phi(\rmd x_{1:\T})$, which is the probability distribution of $X_{1:\T} \eqdef \prange{ X_1 }{ X_{\T} }$ conditioned on
$Y_{1:\T} = y_{1:\T}$.
%
%
Markov chain Monte Carlo can be used to simulate from the \jsd, for example
the Gibbs scheme that uses Metropolis-Hastings samplers to update the state variables $\{ X_t \}_{t=1}^\T$ one-at-a-time,
with all the other variables kept fixed \citep{CarterK:1994,Fruhwirth-Schnatter:1994}.
However, the typically strong dependencies between consecutive states in the state sequence can cause this method to mix very slowly. 
As a result, this solution is often deemed as inefficient.

However, recent developments in \emph{sequential Monte Carlo} (SMC) methods
have had a significant impact on the practise of \mcmc.
The \smc methodology---which combines sequential importance sampling and resampling--- has long been established as a key technique
for approximating the \jsd in general \hmm{s}, see \eg, \citet{DoucetGA:2000,delmoral:2004,DoucetJ:2011}
 for introductions, applications,
and theoretical results. In a seminal paper by \citet{AndrieuDH:2010}, this key strength of \smc was exploited to construct effective (\smc based) high-dimensional \mcmc kernels,
resulting in so called particle \mcmc (\pmcmc) methods.

We will consider specifically an instance of \pmcmc referred to as \emph{particle Gibbs} (\pg)
\citet{AndrieuDH:2010} (see also \citet{ChopinS:2015}).
The \pg algorithm of \citet{AndrieuDH:2010} defines 
(via an \smc construction) a Markov kernel which leaves the full \jsd invariant. From a practitioner's point of view, \pg
can be regarded as a technique to mimic the behaviour of simulating all the state variables $X_{1:\T}$ as one block from the \jsd for models where exact simulation (from the \jsd) is not possible.

There are some recent theoretical results for the \pg (Markov) kernel $\compP$ to back-up its good observed performance.  \citet{ChopinS:2015}
show it is uniformly ergodic but do not shed light on the dependence of the rate on time $\T$ and particle number $\Np$.
The \pg kernel has been shown to be minorised by the \jsd, i.e. $\compP \geq \textrm{const.}\times \phi$, under weak conditions \citep{LindstenDM:2015,AndrieuLV:2015} (thus also
implying uniform ergodicity). Under stronger (but common) forgetting conditions on the \hmm, an explicit lower bound for the minorising constant,
$ (1-\frac{1}{c(\Np-1) +1})^{\T}$,
has also been established in these works; here, 
$1\geq c>0$ is a model-dependent
constant, $\Np$ is the number of particles (of the \smc sampler) used to contruct the \pg kernel and $\T$ is the number of observations. (See \citet{Kuhlenschmidt:2014} for an alternative proof of this rate, and a central limit theorem, using more routine SMC analysis techniques.) 
While the method is indeed uniformly geometrically ergodic for
any $\Np\geq 2$, the explicit rate reveals that the convergence deteriorates as $\T$ increases if $\Np$ is kept fixed.
This effect,
which is also clearly visible in practice \citep{LindstenS:2013,ChopinS:2015},
is related to the well known path-degeneracy issue of \smc samplers, see \eg, \cite{DoucetJ:2011}.
For the rate not to deteriorate, $\Np$ must increase linearly with $\T$, giving rise to an algorithm that costs $\T^2$ per-iteration (or application) of the \pg kernel, which may be impractical when $\T$ is large. (Note that a \pg kernel implemented with $\Np$ particles has a cost per-iteration of $\Np\T$.)

We seek to answer an obvious but important question: is there a different implementation of the \pg kernel with a cost per-iteration that grows
linearly with $\T$, but which is still stable in the sense that its convergence rate does not deteriorate with $\T$? Secondly, and equally important, is the issue of parallelizing the 
\pg kernel to save on wall-clock, or execution, time. We provide an affirmative answer to the posed question and we use blocking to achieve both aims.

\subsection{Summary of main results and related work}
In the literature \pg is often presented (and invoked) as an exact approximation of a \emph{fully blocked} sampler and
it is contrasted with a single-state Gibbs sampler.
However, there is an intermediate route in-between these two extremes, namely to use \pg as a component of a (partially)
blocked Gibbs sampler. (This possibility
was indeed pointed as a potential extension in the seminal \pmcmc paper \citep[p.~294]{AndrieuDH:2010}.) As we will demonstrate in this paper,
using \pg within a blocked Gibbs sampler can result in very efficient samplers. We now give a \emph{simplified} interpretation of our results.

In light of the convergence properties of the \pg kernel discussed above, the convergence rate of the blocked \pg kernel will now depend on the block size $L$ and not on the
total number of observations $\T$. Thus the particle approximation of the ideal blocked Gibbs sampler does not deteriorate with 
total number of observations $\T$ provided the block sizes themselves do not increase with $\T$.
The main insight which we exploit next is that blocking
can also be used to control the convergence properties of the ideal Gibbs sampler for an \hmm. Specifically, we show that under certain forgetting properties of the \hmm it is possible to select
a blocking scheme, using \emph{overlapping} blocks, which results in a uniform rate of convergence for the ideal Gibbs sampler.
Importantly, we show that that the rate of convergence after $k$ complete sweeps is (see Theorems \ref{thm:mr:contraction}, \ref{thm:idealStructuredRates})
$$
\vert \phi \varphi - \mu \compP_{\mathrm{Ideal}}^k \varphi \vert \leq \rate_{\mathrm{Ideal}}^k \sum_{i=1}^\T \osc_i(\varphi)
$$
where $\compP_{\mathrm{Ideal}}$ is the Markov kernel defined by one complete sweep and  the rate $\rate_{\mathrm{Ideal}}$ improves with increasing the overlap between blocks but is independent of $\T$.
Furthermore, we translate this result to the (more practical) case in which exact sampling
from the block-conditionals is not possible and \pg kernels are used to simulate from each of
these conditionals instead. In this case the rate becomes (see \autoref{thm:simplenonidealrates}) $$\rate_{\mathrm{PG}}=\rate_{\mathrm{Ideal}}+\mathrm{const.} \times \epsilon_\Np$$ where 
$0 \leq \epsilon_\Np \leq 1$ quantifies the effect of departing from the ideal block sampler when using the \pg kernels instead. Specifically,  $\epsilon_\Np \downarrow 0$ as the particle number $\Np$ increases, i.e. as the \pg kernel better approximates the ideal block sampler, but $\epsilon_\Np \uparrow 1$ as the size of the blocks increase.
We analyse different blocking schemes (\autoref{thm:simplenonidealrates}) and discuss in particular how the blocking can be selected to
open up for straightforward parallelization of the sampler.

Our analysis is based on Wasserstein estimates, see e.g.\ recent works by \citet{WangW:2014, RebeschiniH:2014}.
\citet{WangW:2014} study the convergence properties of a Gibbs sampler in high dimensions under a Dobrushin condition.
Our work is in the same vein, but we are in particular interested in \emph{verifying} a related condition for
the case of a blocked Gibbs sampler for an \hmm. Furthermore, we study the convergence of the non-ideal blocked Particle Gibbs
sampler, for which the results by \citet{WangW:2014} do not apply.
\citet{RebeschiniH:2014} generalise the Dobrushin comparison theorem and consider applications in high-dimensional filtering.

It should be noted that refined \pg algorithms, incorporating explicit updates of the particle
ancestry either as part of the forward \smc recursion \citep{LindstenJS:2014} or
in a separate backward recursion \citep{Whiteley:2010,WhiteleyAD:2010}, have been developed.
These modified \pg samplers have empirically been shown to work well with small number of particles
and to be largely robust to $\T$. Nevertheless, to date, no theoretical guarantee for the stability
of these algorithms has been given.
For specific blocking schemes, blocked \pg is straightforwardly parallelizable.
It is worth noting that we can replace our blocked \pg kernels with 
the corresponding \pg kernels that  incorporate updates of the particle
ancestry as discussed above.

The rest of this paper is organised as follows.  Section 2 presents the \hmm , the definition of the \jsd, the blocking schemes and particle implementation of the 
ideal blocked Gibbs sampler. Section 3 presents the main results on the uniform ergodicty and rates of convergence of the various samplers. Section 4 presents the proof of the stability of the blocked \pg sampler.
The proofs of supporting technical results are given in the appendices.

\section{Block Sampling and Particle Gibbs}

\subsection{Hidden Markov models}\label{sec:hmm}
Let set $\Xset$ be the state space of the time- and space-homogeneous \hmm. 
We will work under the assumption that a fixed sequence of observations
$y_{1:\T} \eqdef \prange{y_1}{y_\T}$ is available, where $\T$ is some final time point.
The key object of interest
is then the \jsd, that is the probability distribution of $X_{1:\T}$ conditioned on
$Y_{1:\T} = y_{1:\T}$. This will be the target distribution for the various Monte Carlo sampling
algorithms that we shall consider throughout this work.
With,
$$
 p_\Xinitv(y_{1:\T}) \eqdef \int \Xinitv(\rmd x_1) g(x_1,y_1) \prod_{t=2}^\T \kernel{M}{x_{t-1}}{\rmd x_t} g(x_t,y_t) \eqsp,
$$
denoting the density of the observations $Y_{1:\T}$ (with respect to some dominating measure)
we can write the density of the \jsd as
\begin{align}
  \label{eq:jsd-def}
    \phi(x_{1:\T}) \eqdef \frac{1}{p_\Xinitv(y_{1:\T})} \Xinitv(x_1) g(x_1,y_1) \prod_{t=2}^\T \kernel{m}{x_{t-1}}{x_t} g(x_t,y_t)  \eqsp,
\end{align}
where, by abuse of notation, we use the same symbol $\phi$ both for the \jsd and for its density. 
Note that $\kernel{m}{x_{t-1}}{x_t}$ is the density of $\kernel{M}{x_{t-1}}{\rmd x_t}$
\wrt some $\sigma$-finite dominating measure on $\Xset$.
For notational simplicity, we do not make explicit the dependence of $\phi$ on the fixed observation sequence in the notation
in \eqref{eq:jsd-def}.

We will analyse the stability of our samplers under the following set of strong, but standard, mixing assumptions 
\citep{delmoral:2004,LindstenDM:2015,AndrieuLV:2015}

\begin{hyp}{S}
\item \label{asmp:strong-mixing-m}  There exists a probability measure $\nu$ on $\Xset$ such that 
$\kernel{m}{x_{t-1}}{x_t}$ is the density of $\kernel{M}{x_{t-1}}{\rmd x_t}$ \wrt $\nu$.
Furthermore, there exist positive constants $\sigma_{-}$ and $\sigma_{+}$ and
  an integer $\h \in \nset$ such that
  \begin{enumerate}
  \item \label{asmp:strong-mixing-m-plus} $m(x,x') \leq \sigma_{+}$ for all $x, x' \in\Xset$,
  \item \label{asmp:strong-mixing-m-minus} 
  $\int \nu(dx_{2})\ldots \nu(dx_{h}) \left( \prod_{j=1}^{\h}  m(x_{j}, x_{j+1})  \right)  \geq \sigma_{-}$ for all $x_1, x_{\h+1} \in\Xset$.
  \end{enumerate}
  \item \label{asmp:strong-mixing-g} There exists a constant $\delta \geq 1$ such that for all $y \in \Yset$,
    \(\sup_{x} g(x,y) \leq \delta^{1/\h} \inf_x g(x,y) \eqsp.\)
\end{hyp}

\subsection{Block sampling} \label{sec:gibbs}

Before giving the algorithmic statements for the blocked Gibbs samplers that are analysed in this article,
we introduce some notation that will be frequently used in the sequel.
Let $I\eqdef \crange{1}{\T}$ be the ``index set'' of the latent variables $X_{1:\T}$.
Let $\Xset^{\T}$ be the $\T$-fold Cartesian product of the set~$\Xset$ and let the tuple $x= ( x_{i} : i \in I ) \in \Xset^{\T}$. For $J\subset I$, we then write
$x_{J}\eqdef ( x_{i} : i\in J ) \in \Xset^{|J|}$ (the restricted tuple). We also write $x_{-i}$ as a shorthand for $x_{I\setminus \{i\}}$.
The complement of $J$ in $I$ is denoted by $J^c \eqdef I\setminus J$. Given $y= ( y_{i} : i \in J ) \in \Xset^{|J|}$ and $z= ( z_{i} : i \in J^c ) \in \Xset^{|J^c|}$, we
define $x=\join{y}{z}=\join{z}{y} \in \Xset^{\T}$  to be the tuple such that 
$x_J=y$ and $x_{J^c}=z$.

Let $\phi_{x}^{J}$ denote (a version of) the regular
conditional distribution of the variables $X_{J}$ conditionally on $X_{J^c}=x_{J^c}$
under $\phi$ in \eqref{eq:jsd-def}. 
The Markov property of the \hmm implies that $\phi_x^J$ depends on $x$ only through the \emph{boundary points}
$x_{\bp J}$ where $\bp J$ denotes the set of indices which constitutes the boundary of the set $J$
\begin{align*}
  \bp J \eqdef \{ t \in J^c : t+1\in J \text{ or } t-1 \in J \}  \eqsp.
\end{align*}
%
The Gibbs sampler generates samples from the \jsd $\phi$ by iteratively sampling from
its conditional distributions.
Let $\J \eqdef \crange{J_1}{J_m}$ be a cover of $I$.
A \emph{blocked}, deterministic scan Gibbs sampler proceeds by sampling from the block-conditionals
$\phi_{x}^{J}$, $J \in \J$, in turn in some prespecified order.
More precisely, assuming that we apply the blocks in the order $J_1$, $J_2$, etc. and 
if the initial configuration of the sampler is given by $x' \in \Xset^{\T}$, then the Gibbs sampler output is the Markov
chain $\{X[k],\, k\in\nset\} $ with $X[0] = x'$ and, given
\begin{align} \label{eq:sampledprocess}
  X[lm+k-1] = x \in \Xset^{\T}, \quad l\in\nset,\, k\in\crange{1}{m}  \eqsp,
\end{align}
$X[lm+k] = X'$ where
$X'_{J^c_{k}} = x_{J^c_{k}}$ and we simulate
$X'_{J_{k}} \sim \phi_{x}^{J_{k}}(\cdot)$.
More generally, we can simulate $X'_{J}$ from some kernel $\kernel{ Q^J }{ x }{ \cdot }$ with 
the conditional distribution $\phi_x^J$ as its invariant measure. That is, for any $x\in \Xset^{\T}$
\begin{align}
  \label{eq:gibbs:invariance-property}
  &\int \phi_x^J(\rmd x'_J) \kernel{Q^J}{ \join{ x_{J^c}}{ x'_J }}{A} = \phi_x^J(A), &
  &\text{for all measurable $A\subset \Xset^{\card{J}}$} \eqsp.
\end{align}
Since $\phi_x^J$ depends on $x$ only through the boundary points $x_{\bp J}$,
it is natural to assume that $\kernel{ Q^J }{ x }{ \cdot }$ depends on $x$ only through $x_{J\cup\bp J}$.
For notational simplicity, we  define 
\[
\extended{J} \eqdef J \cup \bp J
\]
for any subset $J \subset I$.
Thus for any $(x,z)\in\Xset^{\T}\times\Xset^{\T}$  
we have,
  $\kernel{Q^J}{x}{A} = \kernel{Q^J}{ \join{ x_{\extended{J}} }{ z_{I\setminus\extended{J}} }}{A}$.
We write $\kernel{Q^J}{x_{\extended{J}}}{\rmd x'_J}$ in place of
$\kernel{Q^J}{x}{\rmd x'_J}$ when wishing to 
 to emphasize the dependence of the kernel on the components in $\extended{J}$ of current configuration $x$.
When we have $\kernel{Q^J}{x}{\rmd x'_J} = \phi_x^J (\rmd x'_J )$
for every $J \in \J$, we refer to the sampler as an \emph{ideal Gibbs sampler}.
It follows that the Markov kernel corresponding to updating the block $J$ is given by,
\begin{numcases}{\kernel{P^J}{x}{\rmd x'} \eqdef}
  \phi_x^J (\rmd x'_J ) \times \delta_{x_{J^c}}(\rmd x_{J^c}') \eqsp, & \text{for the ideal Gibbs kernel,} \label{eq:gibbs:P-def}\\
  \kernel{Q^J}{x_{\extended{J}}}{\rmd x_J'} \times \delta_{x_{J^c}}(\rmd x_{J^c}') \eqsp, & \text{for the non-ideal Gibbs kernel.} \label{eq:gibbs:P-def2}
\end{numcases}
Formally, the subsets in $\J$ can be an arbitrary cover of $I$. However,
there are certain blocking schemes that are likely to be of most practical interest
and these schemes therefore deserve some extra attention. To exemplify this, we consider
the following restrictions on the blocks in $\J$.
\begin{hyp}{B}
\item \label{asmp:algs:ordered_blocks}Each $J \in \J$ is an interval, \ie, $J = \crange{s}{u}$ for some
$1\leq s \leq u \leq \T$.
Furthermore, the blocks $\range{J_1}{J_m}$ are ordered 
in the sense that, for any $1\leq j < k \leq m$, $\min(J_j)<\min(J_k)$ and $\max(J_j) < \max(J_k)$.
\item \label{asmp:algs:overlap}Consecutive blocks may overlap, but non-consecutive blocks do not overlap and
are separated. That is, for $1\leq j < k \leq m$ with $k-j \geq 2$, 
$\max(J_j) < \min(J_k) - 1$.
\end{hyp}
In addition to ordering the blocks according to their minimum element, \hypref{asmp:algs:ordered_blocks} avoids the case where one block is a strict subset of some other block. An illustration of \hypref{asmp:algs:overlap}, which requires that $J_{k-1}$ and $J_{k+1}$ do not cover $J_k$, is shown in Figure \ref{fig:algs:example_blocking}.

The Gibbs sampling scheme that (perhaps) first comes to mind is a systematic sweep
from left to right. Under \hypref{asmp:algs:ordered_blocks}, this 
implies that the kernel corresponding to one \emph{complete sweep} of the Gibbs sampler is given by:
\begin{paragraph}{L-R}
Assume \hypref{asmp:algs:ordered_blocks}.
The Left-to-Right Gibbs kernel is given by $\compP  \eqdef P^{J_{1}} \cdots P^{J_{m}}$. \\
\end{paragraph}%


Assuming both \hypref{asmp:algs:ordered_blocks}~and~\hypref{asmp:algs:overlap},
another blocking scheme that is of practical importance is the one that updates
all the odd-numbered blocks first, then all the even-numbered blocks (or the other way around).
We refer to this scheme as the \emph{Parallel} (PAR) Gibbs sampler:
\begin{paragraph}{PAR} Assume \hypref{asmp:algs:ordered_blocks}--\hypref{asmp:algs:overlap}.
The Parallel Gibbs kernel is given by $\compP \eqdef \compP_\text{odd} \compP_\text{even}$ where
\begin{align*}
  &\begin{cases}
    \compP_{\text{odd}} \eqdef P^{J_1} P^{J_3} \cdots P^{J_m}, \\
    \compP_{\text{even}} \eqdef P^{J_2} P^{J_4} \cdots P^{J_{m-1}},
  \end{cases}
  &&\text{or} &
  &\begin{cases}
    \compP_{\text{odd}} \eqdef P^{J_1} P^{J_3} \cdots P^{J_{m-1}}, \\
    \compP_{\text{even}} \eqdef P^{J_2} P^{J_4} \cdots P^{J_{m}},
  \end{cases}
\end{align*}
for odd/even $m$, respectively. \\
\end{paragraph}%

The reason for why we call this sampler \emph{parallel} is that, under
\hypref{asmp:algs:ordered_blocks}-\hypref{asmp:algs:overlap}, 
 it is possible to update all the odd blocks in parallel, followed by a parallel update of all the
even blocks. This is because two consecutive odd (or even) blocks are separated by at least one element in $I$. 
Hence, if we have a total of $C/2$ processing units, we can simulate from both $\compP_\text{odd}$
and $\compP_\text{even}$ (and thus from $\compP$) in constant time for any number of blocks $m\leq C$.
\autoref{fig:algs:example_blocking} shows an example of a blocking configuration which satisfies
\hypref{asmp:algs:ordered_blocks}-\hypref{asmp:algs:overlap}. This is the typical scenario that we have
in mind for the PAR sampler.

\begin{figure}[ptb]
  \centering
  \input{fig-blocking-scheme.tex}%
  \vspace{-2ex}
  \caption{Example block configuration.}
  \label{fig:algs:example_blocking}
\end{figure}
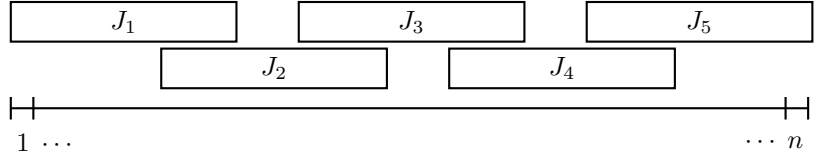

Reversibility in \mcmc is an important consideration. In Section \ref{sec:algs:pg} we use \pg to define the kernels
$\kernel{Q^J}{x_{\extended{J}}}{\rmd x_J'}$ for each $J$ and the \pg kernel is known to be reversible \citep{ChopinS:2015}.
As such it is simple to define reversible block samplers using the following simple fact.

\begin{lemma} 
\label{lem:reverse}
  Let $\J = \crange{J_1}{J_m}$ be an arbitrary cover of $I$ and, for each $J \in \J$, let
  $\kernel{P^J}{x}{\rmd x'} = \kernel{Q^J}{x_{\extended{J}}}{\rmd x_J'} \times \delta_{x_{J^c}}(\rmd x_{J^c}')$
  be the (possibly non-ideal) Gibbs kernel updating block $J$ only. Assume that $Q^J$
  is reversible \wrt to $\phi_x^J$ for all $J \in \J$. Then,
  \begin{align*}
    \phi(\rmd x)P^{J_1}\cdots P^{J_m}(x,\rmd x') = \phi(\rmd x') P^{J_m} \cdots P^{J_1}(x',\rmd x).
  \end{align*}
\end{lemma}
For example,
for the PAR scheme, the kernel $\frac{1}{2}(\compP_{\text{odd}} \compP_{\text{even}} + \compP_{\text{even}} \compP_{\text{odd}} )$
is reversible.
\autoref{lem:reverse} can be used to define other reversible samplers.

\subsection{Particle Gibbs}\label{sec:algs:pg}
Our primary motivation for studying the convergence properties of blocked Gibbs samplers
for general \hmm{s} is the development of \pmcmc. These methods provide systematic and efficient
Markov kernels which can be used to simulate from the block-conditionals $\phi_x^J$ for ``reasonably large'' blocks,
making the blocking strategy practically interesting.
In particular, we will make use of the \pg sampler \citep{AndrieuDH:2010}
to define $Q^J$ and hence the non-ideal Gibbs kernel in~\eqref{eq:gibbs:P-def}.
The \pg sampler of \citet{AndrieuDH:2010} is a Markov kernel that preserves the invariance of the full \jsd.
In this section we briefly review the \pg sampler
and discuss how the standard \pg algorithm needs to be modified (see Algorithm \ref{alg:cpf}) when targeting one of the
block-conditionals $\phi_x^J$ instead of the full joint smoothing distribution $\phi$. 
(Although not detailed as in Algorithm \ref{alg:cpf}, blocking was mentioned by \citet[page~294]{AndrieuDH:2010} as a possible extension
of their original \pg sampler. However, there was no discussion or mention of the greatly improved stability of the sampler
which is our main interest here.)

The \pg sampler can be used to construct a uniformly ergodic Markov kernel on $\Xset^{\card{J}}$
which leaves $\phi_x^J$ invariant. This construction is based on \smc.
 We denote the \pg Markov kernel for block $J$ by $Q^J_\Np$, where $\Np$ denotes
a precision parameter; specifically the number of particles used in the underlying \smc sampler.
We work under assumption \hypref{asmp:algs:ordered_blocks}, \ie that $J$ is an interval.
Furthermore, for the sake of illustration,
assume that $J$ is an internal block, \ie, $J = \crange{\si}{\ui}$ for $\si > 1$ and $\ui < \T$.
Obvious modifications to the algorithmic statement are needed for $\si = 1$ and/or $\ui = \T$.

The conditional \pdf of $X_{\si:\ui}$, given $X_{\si-1} = x_{\si-1}$, $X_{\ui+1} = x_{\ui+1}$, and $Y_{\si:t} = y_{\si:t}$ is
\begin{align}
  \phi_{x}^{\si:u}(x_{\si:u})  
  = \frac{1}{p(y_{\si:u}\mid x_{\si-1}, x_{\ui+1})}
  \left\{ \prod_{j=\si}^u  m(x_{j-1}, x_j) g(x_j, y_j) \right\} m(x_u, x_{u+1}).
\end{align}
To aid the description of the \pg kernel, we define for $t \in J$
the intermediate probability density
\begin{align}
  \pi^{\si:t}_{x_{\si-1}}(x_{\si:t}) \eqdef p(x_{\si:t} \mid x_{\si-1}, y_{\si:t})
  = \frac{1}{p(y_{\si:t}\mid x_{\si-1})}
  \prod_{j=\si}^t  m(x_{j-1}, x_j) g(x_j, y_j),
\end{align}
which is the conditional \pdf of $X_{\si:t}$ given $X_{\si-1} = x_{\si-1}$ and $Y_{\si:t} = y_{\si:t}$.
It follows that the block-conditional density is given by
\begin{align}
  \label{eq:particle-actual-target}
  \phi_{x}^{\si:u}(x_{\si:u}) \propto  \pi^{\si:u}_{x_{\si-1}}(x_{\si:u}) m(x_u, x_{u+1}) . 
\end{align}

The sampling procedure used in the \pg sampler is reminiscent of a standard \smc sampler,
see \eg \citet{DoucetGA:2000,delmoral:2004,DoucetJ:2011,CappeMR:2005}
and the references therein. Here we review a basic method, though it should be noted that the \pg sampler
can be generalised to more advanced methods, see \citet{AndrieuDH:2010,ChopinS:2015}.
As in a standard \smc sampler, the \pg algorithm approximates the sequence of ``target distributions''
$\{\pi^{s:t}_{x_{s-1}},\, t = \range{\si}{\ui} \}$ sequentially by collections of weighted particles.
The key difference between \pg and standard \smc is that, in the former, one particle is set
deterministically according to the current configuration of the sampler.
Let this fixed particle/state trajectory (also called the reference trajectory) be denoted by $x^\star \in \Xset^{\T}$. 
Then, to simulate from the \pg kernel $\kernel{Q_N^J}{x^\star_{\extended{J}}}{\rmd x_J'}$, we proceed as follows.

Initially, we simulate particles $X_{\si}^i \sim r_\si(x_{\si-1}^\star, \cdot)$, $i = \range{1}{\Np-1}$ independently from some proposal density $r_\si(x_{\si-1}, x_{\si})$.
The proposal density may depend on the fixed observation sequence, but we omit
the explicit dependence in the notation for brevity.
Note that all the particles $\{X_{\si}^i\}_{i=1}^\Np$ share
a common ancestor at time $\si-1$, namely the fixed boundary point $x_{\si-1}^\star$.
Note also the we simulate only $\Np-1$ particles in this way.
The $\Np$th particle is set deterministically according to the current configuration: $X_\si^\Np = x_\si^\star$.

To account for the discrepancy between the proposal density and the target density $\pi_{x_{s-1}^\star}^\si$,
importance weights are computed in the usual way: $\omega_{\si}^i = \wf_{\si}(x_{\si-1}^\star, X_{\si}^i)$ for
$i = \range{1}{\Np}$, where the weight function is given by
\begin{align}
  \label{eq:particle-weight-t}
  \wf_{t}(x_{t-1}, x_t) 
  \eqdef
  \frac{g(x_t, y_t) m(x_{t-1}, x_t)}{r_t(x_{t-1}, x_t)} \eqsp.
\end{align}
The weighted particles $\{ (X_\si^i,  \omega_\si^i) \}_{i=1}^\Np$ provide an approximation
of the target distribution $\pi_{x_{\si-1}^\star}^\si$, in the sense that
\begin{align*}
  \widehat \pi_{x_{\si-1}^\star}^\si(f) \eqdef \sum_{i=1}^\Np \frac{ \ewght{\si}{i} }{ \sum_{\ell=1}^N \ewght{\si}{\ell} } f(X_{\si}^i)
\end{align*}
is an estimator of $\int f(x) \pi_{x_{\si-1}^\star}^\si(x)\rmd x$ for any measurable function $f: \Xset \to \rset$.

We proceed inductively.
Assume that we have at hand a weighted sample $\{ (X_{\si:t-1}^i, \ewght{t-1}{i} ) \}_{i=1}^\Np$
approximating the target distribution $\pi_{x_{\si-1}^\star}^{\si:t-1}$ at time $t-1$. This weighted sample is then propagated sequentially \emph{forward in time}. This is done by sampling, for each particle $i \in \crange{1}{\Np-1}$, an \emph{ancestor index} $A_t^i$ with (conditional) probability
\begin{equation}
  \label{eq:ancestor-t}
  \Prb( A_t^i = j )
  = \frac{ \ewght{t-1}{j} }{ \sum_{\ell=1}^N \ewght{t-1}{\ell}} \eqsp, \quad j \in \{1,\dots,N\} \eqsp.
\end{equation}
Given the ancestor, a new particle position is sampled as
$X_{t}^{i} \sim r_t(X_{t-1}^{A_t^i}, \cdot)$,
where $r_t(x_{t-1}, x_t)$ is a proposal density on $\Xset$. Again,
note that we only generate $\Np-1$ particles in this way. The $\Np$th particle and its ancestor index are set
deterministically as $X_t^\Np = x_t^\star$ and $A_t^\Np = \Np$.
The particle trajectories (\ie, the ancestral paths of the particles $X_{t}^{i}$, $i \in \crange{1}{\Np}$)
are constructed sequentially by associating the current particle $X_{t}^i$ with the particle trajectory of its ancestor:
\(
\epart{\si:t}{i} \eqdef (  \epart{\si:t-1}{A_t^i}, \epart{t}{i} ) .
\)
It follows by construction that the $\Np$th particle trajectory coincides with the current configuration:
$\epart{\si:t}{\Np} = x_{\si:t}^\star$ for any $t\in J$.

Finally, the particles are assigned importance weights.
For $t < \ui$, we compute $\omega_{t}^i = \wf_{t}(x_{t-1}^\star, X_{t}^i)$ for
$i = \range{1}{\Np}$, where the weight function is given by \eqref{eq:particle-weight-t}.
At the final iteration of block $J$, we need to take into account the fact that the target distribution
is $\phi_{x^\star}^{\si:\ui}$ and not $\pi_{x_{\si-1}^\star}^{\si:\ui}$; the two being related according to \eqref{eq:particle-actual-target}.
We can view the fixed boundary state $x_{\ui+1}^\star$ as an ``extra observation'' and,
consequently, the weights are computed according to
\begin{align}
  \label{eq:particle-weight-t}
  \ewght{\ui}{i}= 
  \frac{ m(\epart{\ui}{i}, x_{\ui+1}^\star) g(\epart{\ui}{i}, y_\ui) m(\epart{\ui-1}{A_\ui^i}, \epart{\ui}{i}) }{r_\ui (\epart{\ui-1}{A_\ui^i}, \epart{\ui}{i}) }
  =  m(\epart{\ui}{i}, x_{\ui+1}^\star) \wf (\epart{\ui-1}{A_\ui^i}, \epart{\ui}{i})
  \eqsp.
\end{align}
Similarly to the fact that the proposal densities are allowed to depend on the fixed observations $y_{\si:\ui}$,
it is possible to let the final proposal density $r_u$ depend on the fixed boundary point $x_{\ui+1}^\star$,
but again we do not make this explicit in the notation for simplicity.


After a complete pass of the above procedure, the current configuration of the sampler
is updated by replacing states $x_J^\star$ (recall that $J = \si:\ui$) by a draw from the particle approximation of $\phi_{x^\star}^{J}$.
That is, $\epart{J}{\prime}$ is sampled from among the
particle trajectories at time $\ui$,
with probabilities given by their importance weights, \ie,
\begin{equation}
  \label{eq:probability-selection}
  \Prb( \epart{J}{\prime} = \epart{J}{i}  )
  = \frac{\ewght{\ui}{i}}{\sum_{\ell=1}^N \ewght{\ui}{\ell}} \eqsp, \quad i \in \{1,\dots,N\} \eqsp,
\end{equation}
and $x^\star \gets \join{ x^\star_{J^c} }{ \epart{J}{\prime} }$ is taken as the new configuration of the sampler.
The block \pg sampling procedure is summarized in Algorithm~\ref{alg:cpf}.
Note that the procedure associates each $x^\star_{\extended{J}}$ with a probability distribution
on $\Xset^{\card{J}}$; this is the \pg kernel $Q_N^J$.
As shown by \citet{AndrieuDH:2010}, the conditioning on a reference trajectory implies that the \pg kernel leaves the
conditional distribution $\phi_x^J$ invariant in the sense of \eqref{eq:gibbs:invariance-property}.
Quite remarkably, this invariance property holds for any $N\geq 1$ (though, $\Np \geq 2$ is required
for the kernel to be ergodic; see \citet{AndrieuDH:2010,LindstenDM:2015,AndrieuLV:2015}).

{\def\baselinestretch{1.25}
\begin{algorithm}[ptb]
  \caption{\pg sampler for $\phi_x^J$ for $J = \crange{\si}{\ui}$ (non-boundary block)}
  \label{alg:cpf}
  \begin{algorithmic}[1]
    \REQUIRE Observations $\Y_J$, fixed boundary states $x^\star_{\bp{J}}$, and reference states $x_{J}^\star$.
    \ENSURE Draw from the \pg Markov kernel 
    $\kernel{Q^J_\Np}{x^\star_{\extended{J}}}{\rmd x_J^\prime}$.
    \STATE Draw $ \epart{\si}{i} \sim r_\si( x^\star_{\si-1}, \cdot)$ for $i = \range{1}{\Np-1}$
    and set $\epart{\si}{\Np} = x^\star_{\si}$.
    \STATE Set $\ewght{\si}{i} = \wf_{\si}(x^\star_{\si-1}, \epart{\si}{i})$ for $i = \range{1}{\Np}$.
    \FOR{$t = \si+1$ \TO $\ui$}
    \STATE Draw $A_t^{i}$ with $ \Prb ( A_t^{i} = j ) = \ewght{t-1}{j} / \sum_{\ell=1}^{\Np} \ewght{t-1}{\ell}$,
    for  $i =\range{1}{\Np-1}$.%
    \STATE Draw $ \epart{t}{i} \sim r_t(\epart{t-1}{A_t^{i}}, \cdot)$ for $i = \range{1}{\Np-1}$.
    \STATE Set $\epart{t}{\Np} = x^\star_{t}$ and $A_t^{\Np} = \Np$.
    \IF{$t < u$}
    \STATE Set $\ewght{t}{i} = \wf_t( \epart{t-1}{A_t^{i}}, \epart{t}{i})$ for $i = \range{1}{\Np}$.
    \ELSE
    \STATE Set $\ewght{\ui}{i} = m(\epart{\ui}{i}, x^\star_{\ui+1}) \wf_\ui(\epart{\ui-1}{A_\ui^{i}}, \epart{\ui}{i})$ for $i = \range{1}{\Np}$.
    \ENDIF
    \STATE Set $\epart{\si:t}{i} = (\epart{\si:t-1}{A_t^{i}}, \epart{t}{i})$ for $i = \range{1}{\Np}$.
    \ENDFOR
    \STATE Draw $K$ with $ \Prb(K = j) = \ewght{\ui}{j} / \sum_{\ell=1}^{\Np} \ewght{\ui}{\ell}$, $j \in \{1,\dots,\Np\}$.%
    \RETURN $X_{J}^\prime \eqdef X_{J}^{K}$.
  \end{algorithmic}
\end{algorithm}
\def\baselinestretch{1.0}
}

Empirically, it has been found that the mixing of the \pg kernel can be improved significantly by
updating the ancestor indices $A_t^N$, for $t \in \{1,\dots,T\}$, either as part of a separate backward recursion \citep{WhiteleyAD:2010},
or in the the forward recursion
\citep{LindstenJS:2014} itself. 
Although we do not elaborate on the use of these modified \pg algorithms in this work,
the stability results of the blocked \pg sampler presented in the subsequent sections also hold when the \pg kernel
is replaced by one of these modified algorithms (which might result in better empirical performance).

\section{Main Results:  Convergence of the Block Samplers}\label{sec:mainresults}
In this section we state the main convergence results for the blocked Gibbs samplers detailed
in Section~\ref{sec:gibbs}. After introducing our notation and stating some known preliminary results in Section~\ref{sec:mainresults:preliminaries},
we start in Section~\ref{sec:mainresults:wasserstein} by deriving Wasserstein estimates for the (ideal) blocked Gibbs kernels.
We then investigate in Sections~\ref{sec:mainresults:ideal} and \ref{sec:mainresults:pgibbs} how these estimates result in contraction
rates for the ideal and Particle Gibbs block samplers, respectively.


\subsection{Preliminaries, notation, and definitions}\label{sec:mainresults:preliminaries}
For a function $f : \Xset^{\T} \mapsto \rset$, the oscillation of $f$ \wrt\ the
$i$-th coordinate is denoted by
\begin{align}
  \label{eq:bkg:def-lipschitz}
  \osc_{i}(f)=\sup_{\substack{ x,y \in \Xset^{\T}\\ x_{-i}=y_{-i} }} |f(x)-f(y)| \eqsp.
\end{align} while $\osc(f)=\sup_{x, z} |f(x)-f(z)|$ denotes the total oscillation of $f$.
Note that 
\begin{align}
  \label{eq:bkg:bound-osc}
  f(x)-f(z) \leq \sum_{i\in I} \osc_i(f) \I_{[x_i \neq z_i]}.
\end{align}
For a matrix $A$, recall that 
$\| A\| _{\infty}=\max_{i}\sum_{j}|a_{i,j}| = \max_{i}[A\one]_i$ where the last equality holds only if all elements are non-negative. 
The norm is sub-multiplicative, i.e. $\| AB\| _{\infty} \leq\| A\| _{\infty} \| B\| _{\infty} $.
%
%
%
Let $\mu$ and $\nu$ be two probability measures on $\Xset$. With $\Psi$ being a probability measure on $\Xset \times \Xset$
we say that $\Psi$ is a \emph{coupling} of $\mu$ and $\nu$ if $\int \Cpl(\cdot, \rmd x) = \mu(\cdot)$  and $\int \Cpl(\rmd x, \cdot) = \nu(\cdot)$.

We review some well-known techniques for the analysis of Markov chains (see, \eg, \citet{Follmer:1982}). Let $P$ be a Markov kernel on $\XXset$.
The matrix $W$ is a \emph{Wasserstein matrix} for $P$ if for any function $f$ of finite oscillation,
\begin{align}
  \label{eq:defn_wasser}
  \osc_j(Pf) &\leq \sum_{i\in I} \osc_i(f) W_{ij},  & &\text{for all $j \in I$} \eqsp.
\end{align}
If $P$ and $Q$ are two Markov kernels with Wasserstein matrices $V$ and $W$, respectively, then
$WV$ is a Wasserstein matrix for the composite kernel $PQ$.
%

The convergence rate of an \mcmc procedure can be characterised
in terms of a corresponding Wasserstein matrix for the Markov transition kernel, through the following (well-known) result.
For any
probability distributions $\mu$ and $\nu$ and any function $f$ with finite oscillation, we have
\begin{align}
  |\mu P^k f - \nu P^k f| &\leq \sum_{i,j\in I}  \osc_i(f) W_{ij}^k \Cpl(X_j \neq Z_j) \nonumber \\
 & \leq  \left( \sum_{i\in I} \osc_i(f) \right) \| W \|_{\infty}^k \max_{j \in I} \Cpl(X_j \neq Z_j) 
  \label{eq:wassergivescontraction}
\end{align}
for any $k \geq 1$ and for any coupling $\Cpl$ of $\mu$ and $\nu$.
(Note that $\Cpl(X_j \neq Z_j) = \int \Cpl(\rmd x, \rmd z) \I_{[x_j \neq z_j]}$ is the probability of not coupling element $j$ under the coupling $\Cpl$.) 
This result can be verified by using \eqref{eq:bkg:bound-osc} and iterating the inequality \eqref{eq:defn_wasser}.
It follows from \eqref{eq:wassergivescontraction} that $\| W \|_{\infty} < 1$ implies a geometric rate of contraction of the kernel $P$.




\subsection{Wasserstein estimates for the ideal block sampler}\label{sec:mainresults:wasserstein}
Our convergence results, both for the ideal and for the \pg samplers, rely on constructions of Wasserstein matrices for the \emph{ideal} Gibbs kernels.
Indeed, as we shall see in Section~\ref{sec:mainresults:pgibbs}, the \pg block sampler can be viewed
as an $\epsilon$-perturbation of the ideal block sampler, and this is exploited in the convergence analysis.
Let $P^J$ denote the ideal Gibbs kernel that updates block $J$ only, as defined in \eqref{eq:gibbs:P-def},
and let $W^J$ be a Wasserstein matrix for $P^J$. The following lemma reveals the structure of $W^J$.

\begin{lemma}\label{lem:structure-W-general}
  A Wasserstein matrix for the ideal Gibbs kernel updating block $J$ can be chosen to satisfy,
  \begin{align}
    \label{eq:W-What-structure}
    W_{i,j}^J  &=
    \begin{cases}
      \I_{[i=j]}, & i \in J^c, \\
      0, & i \in J,\, j\in I\setminus \partial J, \\
      \times, & i \in J,\, j\in\partial J,
    \end{cases}
  \end{align}
  where  $\times$ denotes elements which are in general in the interval $[0,1]$. 
\end{lemma}
The interpretation of components that are either 0 or 1 is noteworthy as these cannot be improved further. (Recall, the smaller the row sum of the Wasserstein matrix the better.) For example, the lemma states that when the
states $x$ and $z$
differ in a single component, say $x_j \neq z_j$, and if $j \in J$, then this error is not propogated when computing the difference   $P^Jf(x)-P^Jf(z)$. In this sense the given Wasserstein matrices are the `element-wise minimal' ones.


\begin{proof}[Proof (\autoref{lem:structure-W-general})]
From \eqref{eq:gibbs:P-def} we have
\[
\kernel{P^J}{x}{\rmd x'} =
\phi_x^J(\rmd x_J')\delta_{x_{J^c}}(\rmd x_{J^c}')
 \eqsp.
\]
A candidate Wasserstein matrix $W^{J}$
can be found via a coupling argument.
For any pair $(x,z)$ such that $x_{-j}=z_{-j}$ and $x_{j}\neq z_{j}$,
let $\Cpl_{j,x,z}^{J}$ be a coupling of
$\phi^J_x$ and $\phi_z^J$.
Consider now 
\begin{align*}
| P^Jf(x)-P^Jf(z) | &\leq \int \Cpl_{j,x,z}^{J}(\rmd x_{J}',\rmd z_{J}')\delta_{x_{J^c}}(\rmd x'_{J^c})\delta_{z_{J^c}}(\rmd z'_{J^c}) | f(x')-f(z')| \\
&\leq \sum_{i\in I}\osc_{i}(f)\int \Cpl_{j,x,z}^{J}(\rmd x_{J}',\rmd z_{J}')\delta_{x_{J^c}}(\rmd x'_{J^c})\delta_{z_{J^c}}(\rmd z'_{J^c}) \I_{[x_{i}' \neq z_{i}']}\\
&= \sum_{i\in J}\osc_{i}(f)\int \Cpl_{j,x,z}^{J}(\rmd x_{i}',\rmd z_{i}') \I_{[x_{i}'\neq z_{i}']} + \sum_{i\in J^c}\osc_{i}(f) \I_{[x_{i}\neq z_{i}]}\\
&= \sum_{i\in J}\osc_{i}(f)\Cpl_{j,x,z}^{J}( X_{i}'\neq Z_{i}') + \mathbb{I}_{[j\in J^c]}\osc_{j}(f) \eqsp,
\end{align*}
where the second line follows from \eqref{eq:bkg:bound-osc} and the last line follows since $x_{-j}=z_{-j}$, \ie,
at most one term from the second sum in the penultimate line will be non-zero.

Thus, for $i\in J$ we can set
\begin{align}
  \label{eq:BCW:WJdefn}
  W_{i,j}^{J} \eqdef
  \!\!
  \sup_{
    {\scriptsize \begin{array}{c}
        x,z\in \Xset^{\T} \\
        x_{-j}=z_{-j} 
      \end{array}}} 
  \!\!
  \Cpl_{j,x,z}^{J}(X'_i \neq Z'_i).
\end{align}
and for $i\notin J$, $ W_{i,j}^{J} \eqdef \I_{[i=j\in J^c]}$.
Furthermore, since $\phi_x^J$ depends on $x$ only through the boundary points $x_{\bp{J}}$,
it follows that for $j \notin \bp{J}$ the coupling $\Cpl_j$ can be made perfect:
$\Cpl_{j,x,z}(X_i' = Z_i') = 1$ for any $i\in J$.
Therefore, it is evident from \eqref{eq:BCW:WJdefn} that $W_{i,j}^J = 0$ for $i\in J$ and $j\in I\setminus \bp{J}$.
\end{proof}

Specifically, if $J = \si:\ui$ is an interval (\cf \hypref{asmp:algs:ordered_blocks}) it follows that $W^J$ is structured as (blanks correspond to zeros),
\begin{align}
  \label{eq:structure-W-ideal}
  W^J = \left[ {\scriptsize
      \begin{array}{ccccccccc}
        1& & & & & & & & \\
        &\ddots & & & & & & & \\
        && 1& & & & & & \\
        && \times&0 & \cdots& 0& \times& & \\
        && \vdots&\vdots & \ddots& \vdots& \vdots& & \\
        && \times&0 & \cdots& 0& \times& & \\
        && & & & & 1& & \\
        && & & & & & \ddots& \\
        && & & && & & 1
      \end{array} }
  \right] \eqsp,
\end{align}
with obvious modifications for the ``boundary blocks'' ($\si = 1$ and/or $\ui = \T$).
The ``square of zeros'' correspond to rows and columns $i\in J$ and $j\in J$, respectively. The columns of $\times$'s correspond to $W^J_{i,j}$ with $i\in J$ and $j\in \partial J$.


It remains to compute the
non-zero, off-diagonal elements of $W^J$. To this end, we use the strong, but standard, mixing conditions 
\hypref{asmp:strong-mixing-m} and \hypref{asmp:strong-mixing-g}, given in Section~\ref{sec:hmm}.
This allows us to express the ``unknown'' elements of $W^J$ in terms of the mixing coefficients of the model.

\begin{lemma}\label{prop:coupling-HMM}
  Assume \hypref{asmp:algs:ordered_blocks} and let $J = \si:\ui$.
  Assume \hypref{asmp:strong-mixing-m} and \hypref{asmp:strong-mixing-g} and define the constant $\alpha \in [0,1)$ as
  \begin{align}
    \alpha \eqdef 1 - \delta^{ \frac{1-\h}{\h} } {\textstyle \frac{\sigma_-}{\sigma_+} } \eqsp,
  \end{align}
  where  $\delta$, $\sigma_-$, $\sigma_+$, and $\h$ are defined in \hypref{asmp:strong-mixing-m} and \hypref{asmp:strong-mixing-g}.
  Then, $W^J$ defined as in \eqref{eq:W-What-structure} and with 
  \begin{align*}
    W^J_{i,j} =
    \begin{cases}
      \alpha^{\lfloor \h^{-1} (i-j) \rfloor}, & j = \si-1 \text{ (if $\si > 1$)}, \\
      \alpha^{\lfloor \h^{-1} (j-i) \rfloor}, & j = \ui+1 \text{ (if $\ui < \T$)},
    \end{cases}
  \end{align*}
  for $i \in J$ and $j \in \partial J$ is a Wasserstein matrix for the ideal Gibbs block-transition kernel $P^J$.
\end{lemma}
\begin{proof}
  See Appendix~\ref{app:proofs:additional}.
\end{proof}

\subsection{Contraction of the ideal block sampler}\label{sec:mainresults:ideal}

Let
\(
\bp \eqdef \bigcup_{J\in\J} \bp{J}
\)
denote the set of all boundary points.
We start by stating a general geometric convergence result which holds for an arbitrary cover $\J$ of $I$.

\begin{hyp}{A}
\item\label{asmp:mr:lambda-ideal}  For all $J \in \J$,
$  \max_{i \in J\cap \bp} \sum_{j \in \bp{J}} W_{i,j}^J \leq \rate < 1 \eqsp. $
\end{hyp}
\begin{theorem}\label{thm:mr:contraction}
  Assume~\hypref{asmp:mr:lambda-ideal}
  and let $W^J$ satisfy \eqref{eq:W-What-structure} for all $J \in \J$.
  Let $\J = \crange{J_1}{J_m}$ be an \emph{arbitrary  cover} of $I$ and let $\W \eqdef W^{J_m}\cdots W^{J_1}$ be the
  Wasserstein matrix for $\compP = P^{J_1}\cdots P^{J_m}$, \ie, for 
  one complete sweep of the ideal Gibbs sampler.
  Then, for $k \geq 1$, $\| \W^k \|_{\infty} \leq \lambda^{k-1} \| \W \|_{\infty}$.
\end{theorem}
\begin{proof}
  See Appendix~\ref{app:proofs:main}.
\end{proof}


The following result now characterises the convergence of the law of the sampled output of the ideal block sampler.

\begin{corollary} \label{cor:obvious}
  Let $\mu$ and $\nu$ be two probability distributions on $\Xset^{\T}$ and let $\Cpl$ be an arbitrary coupling of $\mu$ and $\nu$.
  Under the same conditions as in \autoref{thm:mr:contraction} we have, for any $k \geq 1$ and any $f$ of finite oscillation,
  \[
  |\mu \compP^k f - \nu \compP^k f| \leq
  \| \W \|_\infty \rate^{k-1} \max_{j\in I} \{ \Cpl(X_j \neq Z_j) \} \sum_{i\in I} \osc_i(f).
  \]
\end{corollary}
\autoref{cor:obvious} verifies two important facts. First is that if we are interested only in the convergence of certain marginals of the sampled process $\{ X[k], k \in \nset\}$ in \eqref{eq:sampledprocess} to the corresponding marginal of the \jsd, then the convergence rate is \emph{independent} of the dimension $\T$ of the \jsd. Secondly, for convergence in total variation norm of the law of $X[k]$ to the full \jsd, we attain a
bound on the error which is $O( \T \rate^k )$. That is, the error grows slowly (linearly) with the dimension $\T$ of the \jsd.

Under the conditions of \autoref{prop:coupling-HMM} 
we can clearly see the benefit of blocking for verifying condition \hypref{asmp:mr:lambda-ideal}.
As an illustration, let $\h=1$ in \ref{asmp:strong-mixing-m}. 
The condition for contraction requires that for any $i \in J \cap \partial$ (assuming $J$ is an internal block),
\begin{align}
  \label{eq:illustration-1}
  \sum_{j \in \partial J} W_{i,j}^J = \alpha^{i-(\si-1) } + \alpha^{(\ui+1)-i} < 1\eqsp,
\end{align}
where $\alpha \in [0,1)$ is defined in \autoref{prop:coupling-HMM}.
First of all, we note that it is possible to ensure $\sum_{j \in \partial J} W_{i,j}^J < 1$ for \emph{any} $i\in J$
by increasing the block size $L \eqdef \ui-\si+1$. Indeed, the maximum of \eqref{eq:illustration-1} for $i\in J$ is attained
for $i = \si$ or $i = \ui$, for which $\sum_{j \in \partial J} W_{i,j}^J = \alpha + \alpha^L$.
Secondly, however, we note that \autoref{prop:coupling-HMM} also reveals the benefit of using \emph{overlapping} blocks.
Indeed, since we only need to control \eqref{eq:illustration-1} for $i \in J \cap \partial$, we can select
the blocking scheme so that the set of boundary points $\partial$ excludes indices $i$ close to the boundary of block $J$.
Consequently, by using overlapping blocks we can control both terms in \eqref{eq:illustration-1}---and thus the overall convergence rate
of the algorithm---by increasing $L$.


\autoref{thm:mr:contraction} assumes no specific structure for $\J$ other than it being a cover. As such, it cannot provide a sharper estimate of the contraction since it caters for all blocking structures. In order to refine the contraction estimate we impose the blocking structure formalised by \hypref{asmp:algs:ordered_blocks} and \hypref{asmp:algs:overlap},
as illustrated in Figure~\ref{fig:algs:example_blocking}, and study the interplay between block size, overlap, and convergence. The theorem below improves the estimate of the decay
of errors per complete sweep from $\rate$ in \autoref{thm:mr:contraction} to $\rate^2$ for the blocking structure of Figure~\ref{fig:algs:example_blocking}.


\begin{theorem} \label{thm:idealStructuredRates}
Assume \hypref{asmp:algs:ordered_blocks},~\hypref{asmp:algs:overlap}, and \hypref{asmp:mr:lambda-ideal}. Then, for any $k\geq 1$: 
\begin{itemize}
\item For the ideal PAR sampler, 
\begin{equation}
\| \mathcal{W}^{k} \|_\infty \leq\| \mathcal{W}\|_\infty \lambda^{2(k-1)}
\label{eq:PARidealrate}
\end{equation}
and $\| \mathcal{W}\|_\infty \leq 2$, where $\lambda$ is defined as in \hypref{asmp:mr:lambda-ideal}.
\item For the ideal L-R sampler,
\begin{equation}
  \| \mathcal{W}^{k}\|_\infty \leq    \| \mathcal{W}\|_\infty \beta^{k-1}
  \label{eq:LRidealrate}
\end{equation}
and $\| \mathcal{W}\|_\infty \leq 1+ \lambda$,
where $\beta = \max_{k\in2:m}\lambda a_{k}+b_{k}$,
$a_{k}=W_{\partial_{+}J_{k-1},\partial_{-}J_{k}}^{J_{k}}$, $b_{k}=W_{\partial_{+}J_{k-1},\partial_{+}J_{k}}^{J_{k}}$.
\end{itemize}
\end{theorem}
\begin{proof}
  See Appendix~\ref{app:proofs:main}.
\end{proof}

\begin{remark}\label{rem:interpretIdeal}
There is parity in the two rates of \autoref{thm:idealStructuredRates} since it can be shown that $\beta\approx\lambda^{2}$. For example let each
block be the same length $L$ and the overlap between all adjacent
blocks $J_{k-1},J_{k}\in\mathcal{J}$ be fixed, $|J_{k-1}\cap J_{k}|=p$.
Under strong mixing \hypref{asmp:strong-mixing-m}--\hypref{asmp:strong-mixing-g},
$b_{k}=\alpha^{\lfloor h^{-1}(L-p) \rfloor},$
$a_{k}=\alpha^{\lfloor h^{-1}(p+1) \rfloor}$,
$\lambda=\alpha^{\lfloor h^{-1}(L-p)\rfloor}+\alpha^{\lfloor h^{-1}(p+1) \rfloor}$
and $\beta/\lambda^{2}\rightarrow1$ as $L$ increases and $p$ is fixed.
\end{remark}


\subsection{Contraction of the Particle Gibbs block sampler}\label{sec:mainresults:pgibbs}
We now turn our attention to the Particle Gibbs block sampler.
In this section we state a 
 main result (\autoref{thm:simplenonidealrates}) that parallels \autoref{thm:idealStructuredRates} for the \pg kernel. 
For the sake of interpretability, we specialize the result to the case of a common block size $\fixL$ and common overlap $\p$ between successive blocks (see also \autoref{rem:interpretIdeal}). A version of \autoref{thm:simplenonidealrates} without 
this assumption, nor strong mixing, is presented in Section~\ref{sec:maingeneralresults:pgibbs}; see
Theorems~\ref{thm:parallelPgibbsGeneral}~and~\ref{thm:sequentialPgibbsGeneral}.
(\autoref{thm:simplenonidealrates} is a corollary of these theorems.)
\begin{hyp}{B}
\item For all $J\in\J$, $|J|=\fixL$ and for all consecutive $J_{k-1}, J_k\in\J$, $|J_{k-1} \cap J_k| = \p$. \label{asmp:commonblocks}
\end{hyp}

Note that \hypref{asmp:commonblocks} implies that $\T$ is assumed to satisfy $\T = (L-p)m + p$.

\begin{theorem} \label{thm:simplenonidealrates}
Assume \hypref{asmp:algs:ordered_blocks}--\hypref{asmp:commonblocks} and \hypref{asmp:strong-mixing-m}--\hypref{asmp:strong-mixing-g}.
Let $\compP$ denote the Markov kernel corresponding to one complete sweep of either the PAR sampler or the L-R sampler,
and assume that each block is updated by simulating from the PG kernel (Algorithm~\ref{alg:cpf})
using a bootstrap proposal: $r(x,x') = m(x,x')$.
Let $\mu$ and $\nu$ be two probability distributions on $\Xset^{\T}$ and let $\Cpl$ be an arbitrary coupling of $\mu$ and $\nu$.
Then, for any $f$ of finite oscillation and any $k \geq 1$
\begin{equation}
|\mu \compP^k f - \nu \compP^k f| \leq
\rate_\mathrm{PG}^{k} \times \Cpl(X \neq Z) \sum_{i\in I} \osc_i(f),
\label{eq:thm:simplenonidealrates}
\end{equation}
where:
\begin{itemize}
\item For the Particle Gibbs PAR sampler,
\begin{equation}
  \label{eq:thm:simplenonidealrates_2}
  \rate_{\mathrm{PG}} \leq
  \rate(\maxsumW \vee 1)+\epsilon \left(2\rate +25\epsilon +8(\maxsumW \vee 1)\right)  .
\end{equation}
\item For the Particle Gibbs L-R sampler,
\begin{equation}
\rate_{\mathrm{PG}} \leq \rate + \alpha^{\lfloor h^{-1} (\fixL-\p+1) \rfloor }+ 2\epsilon\frac{3(\maxsumW \vee 1)+1+\rate}{1-2\epsilon-\alpha^{\lfloor  h^{-1}(\p+1) \rfloor}}
\end{equation} provided $2\epsilon+\alpha^{\lfloor  h^{-1} (\p+1) \rfloor}<1$,
\end{itemize}
provided that $\lambda < 1$.
In the above,
\begin{align}
\rate &= 2 \alpha^{\lfloor h^{-1} (\p + 1) \rfloor}, &
\maxsumW &= \alpha^{\lfloor \h^{-1} \rfloor } + \alpha^{\lfloor  \h^{-1} (\fixL - \p +1) \rfloor },
\end{align}
and
\(
  \epsilon = \epsilon(\Np,L) = 1- (1-\frac{1}{c(N-1)+1})^L,
\)
for some constant $c$ (specified in Proposition~\ref{lem:pgibbs-ergodicity} below) which is independent of $\T$, $\Np$, $L$ and $p$.
\end{theorem}
\begin{proof}
  The proof is given in \autoref{sec:maingeneralresults:pgibbs}.
\end{proof}


\autoref{thm:simplenonidealrates} also applies to the ideal sampler (set $\epsilon=0$). In terms of sufficiency for contraction,  the requirement that the non-$\epsilon$ terms of this theorem be less than 1 is stronger than \hypref{asmp:mr:lambda-ideal}; this should not be surprising since the analysis is catered for the 
non-ideal PG kernel
and thus is inherently more conservative.
For common blocks lengths $\fixL$ and overlap  $\p$, \hypref{asmp:mr:lambda-ideal} requires $\alpha^{\lfloor \frac{\p+1}{\h} \rfloor }+\alpha^{\lfloor \frac{\fixL-\p}{\h} \rfloor }<1$.
The non-$\epsilon$ terms of \autoref{thm:simplenonidealrates} are shrunk by increasing the overlap of blocks and then increasing block size with overlap fixed.
 Alternatively, if $p$ is a constant fraction of $L$, then $\rate$ tends to zero as $L$ increases.
We remark that the non-$\epsilon$ rates of \autoref{thm:simplenonidealrates} are not as sharp as they could be as given in
\autoref{thm:parallelPgibbsGeneral} and \autoref{thm:sequentialPgibbsGeneral} below.
The upper bounds for the non-$\epsilon$ terms were chosen for the sake of simplicity and interpretability since they are functions of system forgetting ($\alpha$ and $\h$), common block lengths $\fixL$, and overlap $\p$ only.

\section{Proof of the Contraction of the Particle Gibbs Block Sampler}\label{sec:maingeneralresults:pgibbs}
This section is dedicated to the proof of \autoref{thm:simplenonidealrates}. Furthermore, it provides a more general version of \autoref{thm:simplenonidealrates} that
avoids the common block length and overlap structure of assumption \hypref{asmp:commonblocks}.
The general results given below are stated in terms of Wasserstein estimates for the \emph{non-ideal} blocked Gibbs sampler $\widehat{W}$,
under the assumption \hypref{asmp:WJperturbed-1} that $\widehat W$ is an $\epsilon$-perturbation of a Wasserstein matrix for the ideal block sampler.
Specifically, let $\widehat W^J$ be a Wasserstein matrix for the non-ideal block-transition kernel defined in \eqref{eq:gibbs:P-def2}.
By an analogous argument as in Lemma~\ref{lem:structure-W-general} it follows that
$\widehat W^J$ has a similar structure as $W^J$, but with
(possibly) non-zero entries also for rows and columns $i\in J$ and $j\in J$, respectively, which motivates the following assumed structure.

\begin{hyp}{A}
\item 
  \label{asmp:WJperturbed-1}
  There exists a (common) constant $\epsilon \in [0,1)$ and, for each $J \in \J$, a matrix $W^J$ satisfying \eqref{eq:W-What-structure}
  such that, for any $J \in \J$,
  \(
  \widehat{W}_{i,j}^{J} \eqdef W_{i,j}^{J} + \epsilon\,\mathbb{I}_{[i\in J,j\in J_{+}]}
  \)
  is a Wasserstein matrix for the non-ideal block-transition kernel updating block $J$.
\end{hyp}

In Proposition~\ref{lem:pgibbs-ergodicity} below we show that \hypref{asmp:WJperturbed-1} indeed holds for the PG kernel,
with $W^J$ being a Wasserstein matrix for block $J$ of the ideal sampler (as in \autoref{prop:coupling-HMM})
and where the perturbation $\epsilon$ depends on (grows with) block size $|J|$ and (decreases with) the number of particles $\Np$, but is independent of the length of the total data record $\T$ or the specific \hmm observations pertaining to each block. This is key to the stability of the PG block sampler as $\T\rightarrow \infty$ for fixed $\Np$.

\begin{theorem}
\label{thm:parallelPgibbsGeneral}
Assume \hypref{asmp:algs:ordered_blocks},~\hypref{asmp:algs:overlap}, \hypref{asmp:WJperturbed-1}, PAR and assume that the number of
blocks $m = \card{\J}$ is odd.\footnote{Clearly, an analogous result holds for even $m$.} 
Let $J_{-k}=\cup_{J\in \{ \J\setminus J_{k} \}}J$ and $L=\max_{J\in\J}|J|$. Then 
$\widehat{\W}$, the Wasserstein matrix of one complete sweep (defined analogously to \autoref{thm:mr:contraction},) satisfies
\[
[\widehat{\mathcal{W}}\mathbf{1}]_{i}\leq
\begin{cases}
  \rate^{2}+\epsilon\left(\rate(L+4)+\epsilon\left(L+2\right)^{2}+L(1\vee\maxsumW)\right)  ,
  & i \in  J_{-k}^{c}, \, k\text{ even}  ,\\
  \rate+\epsilon\left(L+2\right)  ,
  & i\in  J_{-k}^{c}, \, k\text{ odd}  ,\\
  \rate\maxsumW+\epsilon\left(\maxsumW(L+2)+2\rate+\epsilon\left(L+2\right)^{2}+L(1\vee\maxsumW)\right)  ,
  & i\in J_{k}\cap J_{-k}, \, k\text{ even} ,
\end{cases}
\]
where 
\begin{eqnarray*}
\rate &=& \max_{J_{k}\in\J}\max_{i\in  J_{-k}^{c}}W_{i,\partial_{-}J_{k}}^{J_{k}}+W_{i,\partial_{+}J_{k}}^{J_{k}},\\
\maxsumW &=& \max_{J\in\J}\max_{i\in J}\sum_{j\in\partial J}W_{i,j}^{J}
\end{eqnarray*}
and $W^{J_1}_{i,\partial_{-} J_1} = W^{J_m}_{i,\partial_{+} J_m } = 0$ by convention.
\end{theorem}
\begin{proof}
See Appendix~\ref{app:proofs:main}.
\end{proof}

\begin{theorem} \label{thm:sequentialPgibbsGeneral}
Assume \hypref{asmp:algs:ordered_blocks},~\hypref{asmp:algs:overlap}, \hypref{asmp:WJperturbed-1} and L-R.
Let $J_{-k}=\cup_{J\in \{ \J\setminus J_{k} \}}J$, $L=\max_{J\in\J}|J|$
and $L_{1}=\max_{k \in 2:m}|J_{k-1}\cap J_{k}|$. 
Then $\widehat{\mathcal{W}}$ satisfies
\[
[\widehat{\mathcal{W}}\mathbf{1}]_{i}\leq
\begin{cases}
  \rate+c\epsilon, & i\in J_{-k}^{c},\, k\in1:m,\\
  \ratetwo+2c\epsilon, & i\in J_{k-1}\cap J_{k},\, k\in2:m,
\end{cases}
\]
where,
\begin{alignat*}{1}
  \rate & =\max_{J_{k}\in\J}
  \max_{i\in  J_{-k}^{c}}
  \frac{W_{i,\partial_{+}J_{k}}^{J_{k}}}{1-W_{i,\partial_{-}J_{k}}^{J_{k}}} ,\\
  \ratetwo & = \max_{k\in 2:m}\max_{i\in J_{k-1}\cap J_{k}}\rate W_{i,\partial_{-}J_{k}}^{J_{k}}+W_{i,\partial_{+}J_{k}}^{J_{k}} ,\\
  c &= \frac{L(\maxsumW(\rate\vee1)\vee1)+1+\rate}{1-(L_{1}+1)\epsilon-\maxW},\\
  \maxsumW & =\max_{J\in\J}\max_{i\in J}\sum_{j\in\partial J}W_{i,j}^{J} ,\\
  \maxW & =\max_{k\in2:m}\max_{i\in J_{k}\cap J_{k-1}^{c}}W_{i,\partial_{-}J_{k}}^{J_{k}} ,
\end{alignat*}
provided that 
$\maxW + (L_{1}+1)\epsilon < 1$. (By convention set $W^{J_1}_{i,\partial_{-} J_1} = W^{J_m}_{i,\partial_{+} J_m } = 0$.)
\end{theorem}
\begin{proof}
See Appendix~\ref{app:proofs:main}.
\end{proof}

In order to prove Theorem~\ref{thm:simplenonidealrates} we can now make use of
\autoref{prop:coupling-HMM} (assuming \hypref{asmp:strong-mixing-m}--\hypref{asmp:strong-mixing-g} and \hypref{asmp:algs:ordered_blocks}--\hypref{asmp:commonblocks})
to identify the constants of Theorems~\ref{thm:parallelPgibbsGeneral} and \ref{thm:sequentialPgibbsGeneral}. 
However, a technical detail is to handle the dependence on the maximum block size $L=\max_{J\in\J}|J|$ and maximum overlap
$L_{1}=\max_{k \in 2:m}|J_{k-1}\cap J_{k}|$ (of \autoref{thm:sequentialPgibbsGeneral}). Indeed, a direct application of Theorems~\ref{thm:parallelPgibbsGeneral} and \ref{thm:sequentialPgibbsGeneral}
would suggest that the norm of $\widehat \W$ grows, respectively, quadraticly or linearly with $\epsilon L$.
To avoid this issue we will make use of the following trick: when applying Theorems~\ref{thm:parallelPgibbsGeneral} and \ref{thm:sequentialPgibbsGeneral}
we do not consider the original \hmm formulation, but an equivalent model that lumps consecutive states together (for the sake of the analysis),
thus effectively reducing the size of the blocks. Under \hypref{asmp:algs:ordered_blocks}--\hypref{asmp:commonblocks},
each block can be split into three distinct sections which are the left overlap, the middle of the block, and the right overlap. (See \autoref{fig:algs:example_blocking}.)  The exception are the end blocks which are split into two sections. By viewing each of these
parts as a single lumped state, we reduce the block size to $3$ and maximum overlap to $1$. 
For this scheme, the lumped states are given by
\begin{align*}
  \Xi_1 &= X_{1:L-p}, &\Xi_2 &= X_{L-p+1:L}, &   \\
 \Xi_3 &= X_{L+1:2L-2p}, & \Xi_4 &= X_{2L-2p+1:2L-p},& \Xi_5 &= X_{2L-p+1:3L-3p}, \\
 & \vdots 
\end{align*}
where $\Xi_1$ and $\Xi_2$ are the states of the two sections of block $1$, $\Xi_2$, $\Xi_3$ and $\Xi_4$ are the states of block $2$ etc. More precisely:
\begin{remark}
The $\Xi$-system groups the random variables $X_1,\ldots,X_{\T}$ of the $X$-system as $\Xi_{1},\ldots,\Xi_{2m-1}$, noting that 
\hypref{asmp:commonblocks} implies that $n = (L-p)m + p$, where
\begin{align*}
 \begin{cases}
   \Xi_1 = X_{1:L-p}, & \\
   \Xi_{2i} = X_{(L-p)i+1 : (L-p)i+p}, & 1 \leq i < m, \\
   \Xi_{2i-1} = X_{(L-p)(i-1)+p+1 : (L-p)i}, & 1 < i < m, \\
   \Xi_{2m-1} = X_{(L-p)(m-1)+p+1 : (L-p)m+p}.
 \end{cases}
\end{align*}
The index set for the $\Xi$-system is $I_{\Xi}=\{1,\ldots,2m-1\}$ and the cover $\J_\Xi$ of $I_{\Xi}$ has $m$ sets with (set $k$) $J_{\Xi,k} = \{2k-2, 2k-1, 2k\}\cap I_{\Xi}$.
\label{rem:Xisysdefn}
\end{remark}

To find a Wasserstein matrix for the $\Xi$-system we note that any conditional density of the states $\Xi_i$, $i\in J_{\Xi,k}$, conditionally on the
boundaries of the block $J_{\Xi,k}$ and the observation pertaining to that block,
is coupled analogously to the $X$-system; see the proof of \autoref{prop:coupling-HMM}
in the appendix.
Analogously to \autoref{prop:coupling-HMM}, a Wasserstein matrix for block $J_{\Xi,k}$ of the $\Xi$-system is thus given by the $(2m-1) \times (2m-1)$ matrix
\begin{align}
  \label{eq:lumped-W-ideal}
  W_{\Xi}^k = \left[ {\scriptsize
      \begin{array}{ccccccccc}
        1& & & & & & & & \\
        &\ddots & & & & & & & \\
        && 1& & & & & & \\
        && \alpha^{\lfloor h^{-1} \rfloor} &0 & 0& 0& \alpha^{\lfloor h^{-1} (L-p+1) \rfloor} & & \\
        && \alpha^{\lfloor h^{-1}(p+1) \rfloor} & 0 & 0& 0& \alpha^{\lfloor h^{-1}(p+1) \rfloor} & & \\
        && \alpha^{\lfloor h^{-1} (L-p+1) \rfloor} &0 & 0& 0& \alpha^{\lfloor h^{-1} \rfloor} & & \\
        && & & & & 1& & \\
        && & & & & & \ddots& \\
        && & & && & & 1
      \end{array} }
  \right] \eqsp,
\end{align}
where the ``square of zeros'' correspond to rows/columns $2k-2$, $2k-1$, and $2k$ (with obvious modifications for $k = 1$ or $k = 2m-1$).
The proof is omitted for brevity, but follows analogously to the proof of \autoref{prop:coupling-HMM}.

As a final ingredient to prove Theorem~\ref{thm:simplenonidealrates} we need to verify condition \hypref{asmp:WJperturbed-1}
for the \pg kernel.

\begin{proposition}
  \label{lem:pgibbs-ergodicity}
  Assume \hypref{asmp:strong-mixing-m}, \hypref{asmp:strong-mixing-g} and assume that the bootstrap proposal kernel
  is used in Algorithm~\ref{alg:cpf}: $r(x,x') = m(x,x')$. Then,
  \begin{enumerate}
  \item For any $J = \crange{\si}{\ui} \subseteq I$ with $\ui \geq \si$ and any $\Np \geq 2$,
    \[
    \kernel{Q_N^J}{x_{J_{+}}}{\rmd x_{J}'} \geq (1-\epsilon(\Np,|J|)) \eqsp \phi^J_{x}(\rmd x_{J}')
    \]
    where $\epsilon(\Np,L) = 1- (1-\frac{1}{c(N-1)+1})^L$,
    \[
    c = \left( 2 \delta \frac{\sigma_{+}}{\sigma_{-}} - 1 \right)^{-1} \in (0,1]
    \]
    and $\delta, \sigma_{-}$ and $\sigma_{+}$ are defined in
    \hypref{asmp:strong-mixing-m}, \hypref{asmp:strong-mixing-g}.
  \item Assume further \hypref{asmp:algs:ordered_blocks}--\hypref{asmp:commonblocks}. For $k = \range{1}{m}$, let
    Let $W^k_\Xi$ be given by \eqref{eq:lumped-W-ideal} (\ie, $W^k_\Xi$ is a Wasserstein matrix, based on the $\Xi$-system, for the ideal Gibbs kernel updating block~$J_k$).
    Then $\widehat W_\Xi^k = [\widehat{W}^k_{\Xi,i,j}]_{i,j\in 1:2m-1}$ with
    \begin{align}
      \label{eq:lumped-W-epsilon}
      \widehat{W}^k_{\Xi,i,j} = W^k_{\Xi,i,j} + \epsilon(N,|J_k|) \mathbb{I}_{\left[{\scriptsize \begin{array}{c} 2k-2 \leq i \leq 2k \\ 2k-3 \leq j \leq 2k+1 \end{array} }\right]}
    \end{align}
    is a Wasserstein matrix for the \pg kernel $Q_N^{J_k}$. 
  \end{enumerate}
\end{proposition}
\begin{proof}
  See Appendix~\ref{app:proofs:additional}.
\end{proof}

\begin{remark}
  We emphasise that the lumping of state variables is used only for the sake of analysis,
  to improve the contraction rates in Theorems~\ref{thm:parallelPgibbsGeneral} and \ref{thm:sequentialPgibbsGeneral}
  by avoiding a dependence on the block size.
  For all practical purposes the lumping has no effect: when implementing the \pg kernel (Algorithm~\ref{alg:cpf}) we still use
  the original state variables. Consequently, lumping does not affect the ergodicity or the convergence rate of the \pg kernel.
  In particular, note that $\epsilon$ in \eqref{eq:lumped-W-epsilon} depends on the size of block $k$
  as expressed in the $X$-system (\,$|J_k|$) and not in the lumped $\Xi$-system (\,$|J_{\Xi,k}|$).
\end{remark}

We conclude this section by putting the pieces together to prove \autoref{thm:simplenonidealrates}.
\begin{proof}[Proof (\autoref{thm:simplenonidealrates})]
\autoref{thm:simplenonidealrates} is established by applying Theorems~\ref{thm:parallelPgibbsGeneral} and \ref{thm:sequentialPgibbsGeneral} to the $\Xi$-system.
The bound \eqref{eq:thm:simplenonidealrates} is established by using \eqref{eq:wassergivescontraction} for the $\Xi$-system:
\[
|\mu \compP^k f - \nu \compP^k f| \leq
\left( \sum_{i=1}^{2m-1} \sup_{ \{ \Xi,\check{\Xi} \in \Xset^{\T} : \Xi_{-i} =  \check{\Xi}_{-i} \} } | f(\Xi) - f(\check{\Xi})| \right)
\| \W_\Xi \|_\infty^{k}
\max_{j \in 1:2m-1} \Cpl({\Xi}_j \neq \check{\Xi}_j).
\]
Since $(\Xi_1,\ldots,\Xi_{2m-1})\in \Xset^{\T}$, we can crudely bound
the sum 
with $\sum_{i=1}^{\T} \sup_{\{x,y\in \Xset^{\T}: x_{-i}=y_{-i}\}} |f(x)-f(y)|$.
The final factor 
is also crudely bounded by
$\max_{j \in I} \Cpl({\Xi}_j \neq \check{\Xi}_j) \leq \Cpl({\Xi} \neq \check{\Xi})$, for any coupling $\Cpl$ of $\mu$ and $\nu$ with $(\Xi,\check{\Xi})\sim \Cpl$.

Now it remains to derive the expression for $\rate_{\mathrm{PG}}$ as an upper bound on $\|\W_\Xi \|_\infty$
for both cases PAR and L-R.
We detail the case PAR only, as L-R follows by analogous arguments. For case PAR we use \autoref{thm:parallelPgibbsGeneral} for the $\Xi$-system. 
\autoref{thm:parallelPgibbsGeneral} is applicable for the $\Xi$-system since the $\Xi$-system
satisfies \hypref{asmp:algs:ordered_blocks},~\hypref{asmp:algs:overlap} (see \autoref{rem:Xisysdefn}) and \hypref{asmp:WJperturbed-1}. The fact that the $\Xi$-system satisfies \hypref{asmp:WJperturbed-1} follows from \autoref{lem:pgibbs-ergodicity}.
Each block $J_{\Xi,k}$ of the $\Xi$-system has $3$ elements (except the initial and final blocks which both have $2$ elements each);
see \autoref{rem:Xisysdefn}. The specific values of $\rate$ and $\maxsumW$ of \autoref{thm:parallelPgibbsGeneral}
follow from this simple $3$-element block structure and the declared Wasserstein matrix in \eqref{eq:lumped-W-ideal}. The coefficient of the $\epsilon$-term in \eqref{eq:thm:simplenonidealrates_2} follows from a trivial bound of the three separate $\epsilon$-coefficients given in \autoref{thm:parallelPgibbsGeneral} using $L=3$.
 \end{proof}





\appendix
\section{Proofs of the Main Theorems} \label{app:proofs:main}

\begin{lemma} \label{lem:mask}
  Let $\J = \{J_1, \dots, J_m\}$ be an arbitrary cover of $I$. For each $J\in\J$, let $W^J$ be a matrix with structure \eqref{eq:W-What-structure}
  and let $\W = W^{J_m}\cdots W^{J_1}$.
  For $j \in I$, define, \(    a_j \eqdef \min \{k  : j \in J_k\} \) and \(  b_j = \min \{k  : j \in \partial J_k\}\),
  with the convention that $\min\{\emptyset\} = \infty$.
  If $a_j < b_j$, then $\W e_j = 0$.
\end{lemma}
\begin{proof}
  Due to the structure of $W^J$ from \eqref{eq:W-What-structure} it follows that, for $j \notin \partial J$,
  \(
  W^J e_j = \I_{[j \in J^c]} e_j.
  \)
  If $a_j < b_j$, then 
  \begin{align*}
  \W e_j &= W^{J_m} \cdots W^{J_{b_j}} \cdots W^{J_{a_j}} \cdots W^{J_1} e_j \\
  &= 
  W^{J_m} \cdots W^{J_{b_j}} \cdots W^{J_{a_j}} e_j = 0,
  \end{align*}
  where the first equality follows from the fact that for any $k < a_j < b_j$, $j \notin J_k \cup \partial J_k$ by the definition of $a_j$ and $b_j$.
\end{proof}

\subsection{Proof of \autoref{thm:mr:contraction}}

From \autoref{lem:mask} it is clear that $\W e_j = 0$ for any $j \notin \partial$ (for which $b_j = \infty$).
Hence, with $M$ being the binary mask matrix with elements $M_{i,j}=\mathbb{I}_{[i=j\in\partial]}$, it holds that
 $\mathcal{W}r=\mathcal{W}(Mr)$ for any vector $r$.
Thus $\mathcal{W}\mathcal{W}^{k-1}r=\mathcal{W}(M\mathcal{W})^{k-1}r$
and hence $\| \mathcal{W}^{k}\|_\infty \leq \| \mathcal{W}\|_\infty \| M\mathcal{W}\|_\infty^{k-1}$.

Next, we consider
\(
  \| M\W \|_\infty = \max_{i\in\bp} \left\{ e_i^\+ \W \one \right\}.
\)
Define recursively, $L_0 \eqdef \emptyset$ and  $r_0 \eqdef \one$ and
\begin{align}
  \label{eq:pf:Lk-rk-def}
  \begin{cases}
    L_k \eqdef L_{k-1} \cup J_k, \\
    r_k \eqdef W^{J_k} r_{k-1},
  \end{cases} \text{for } k=\range{1}{m} .
\end{align}
We thus have $r_m = \W \one$.
Hence, the result follows if, for $k = \range{0}{m}$,
\begin{align}
  \label{eq:pf:rk-result}
  [r_k]_i \leq
  \begin{cases}
    1 & \text{if } i \in I\setminus L_k, \\
    \rate & \text{if } i \in L_k \cap \bp,
  \end{cases}
\end{align}
\noindent (note that nothing is said about $[r_k]_i$ for $i \in L_k \setminus \bp$).
Indeed, the fact that $\J$ is a cover of $I$ implies that $L_m = I$. Hence $[\W \one]_i = [r_m]_i \leq \rate$ for all $i \in \bp$.

It remains to prove \eqref{eq:pf:rk-result}.
For $k = 0$ the hypothesis is true by construction. We proceed inductively. Hence, assume that
the hypothesis is true for $k-1$. Consider
\begin{align}\label{eq:pf:r_k}
  [r_k]_i = [W^{J_k} r_{k-1}]_i =
  \begin{cases}
    [r_{k-1}]_i & \text{if } i \in J^c_k, \\
    \sum_{j\in\bp{J_k}} W_{i,j}^{J_k} [r_{k-1}]_j & \text{if } i \in J_k,
  \end{cases}
\end{align}
where we have made use of the structure of the Wasserstein matrix from Lemma~\ref{lem:structure-W-general}.
We need to consider different cases:
\begin{enumerate}
\item $i \in I\setminus L_k \Rightarrow [r_k]_i = [r_{k-1}]_i \leq 1$, where we first use \eqref{eq:pf:r_k}
  and then the induction hypothesis, and the fact that $L_k = L_{k-1} \cup J_k$.
\item $i \in ( L_{k-1}\setminus J_k ) \cap \bp \Rightarrow [r_k]_i = [r_{k-1}]_i \leq \rate$, where, again, we use
  \eqref{eq:pf:r_k} and the induction hypothesis.
\item $i \in J_k \cap \bp \Rightarrow
  [r_k]_i = \sum_{j\in\bp{J}} W_{i,j}^{J_k} [r_{k-1}]_j \leq \sum_{j\in\bp{J}} W_{i,j}^{J_k} \leq \rate$, where we use
  the fact that for $j \in \bp$, $[r_{k-1}]_j \leq 1$ and assumption~\hypref{asmp:mr:lambda-ideal} for the final inequality.
\end{enumerate}
This completes the proof.
\hfill $\qedsymbol$

\subsection{Proof of \autoref{thm:idealStructuredRates}}
\textbf{Case L-R:} The proof is similar to that of \autoref{thm:mr:contraction}, but to exploit the structure of the L-R
sampler we define the mask matrix $M$ as $M_{i,j}=\mathbb{I}_{[i=j\in\partial_{+}]}$
where $\partial_{+}=\cup_{J\in\mathcal{J}}\partial_{+}J$ is the set
of \emph{right boundary points}, only, of all blocks. If $j = \partial_{-} J_k$ is the left boundary of some block $k$, say, then
$j \in J_{k-1}$ by definition of the L-R sampler. Hence, from \autoref{lem:mask}
it follows that
$\mathcal{W}r=\mathcal{W}(Mr)$ for any vector $r$.
Thus, $\| \mathcal{W}^{k}\|_\infty \leq \| \mathcal{W}\|_\infty \| M\mathcal{W}\|_\infty ^{k-1}$.

Define $L_k$ and $r_k$ as in \eqref{eq:pf:Lk-rk-def}. Note that \eqref{eq:pf:rk-result} and \eqref{eq:pf:r_k}
hold for any cover $\J$, and in particular for the L-R sampler. Thus, for $i \in J_k$,
\begin{align}
  \label{eq:pf:rk-lr}
  [r_k]_i =  W_{i,\partial_{-} J_k}^{J_k} [r_{k-1}]_{\partial_{-}J_k} + W_{i,\partial_{+} J_k}^{J_k} [r_{k-1}]_{\partial_{+}J_k}
  \leq  W_{i,\partial_{-} J_k}^{J_k} \lambda + W_{i,\partial_{+} J_k}^{J_k} \eqsp,
\end{align}
where the inequality follows from the fact that, for the L-R sampler, $\partial_{-} J_k \in L_{k-1}$ and $\partial_{+} J_k \in I \setminus L_{k}$
and by using \eqref{eq:pf:rk-result}.
Hence, $[r_k]_{\partial_{+}J_{k-1}} \leq \beta$.
Since, for any $k \in 1:m$, $\partial_{+} J_{k-1}$ lies in at most one block (namely $J_k$), we conclude that $[r_m]_i \leq \beta$ for any
$i \in \partial_{+}$. Thus $\| M\mathcal{W}\|_\infty \leq \beta$.
The bound $\| \mathcal{W}\|_\infty \leq 1+ \lambda$ follows similarly from \eqref{eq:pf:rk-lr}.

\textbf{Case PAR:} 
For the PAR sampler we
redefine the mask $M$ to be $M_{i,j}=\mathbb{I}_{[i=j\in\partial_{odd}]}$,
where $\partial_{odd}$ is the set of all boundary points of odd blocks,
\ie, $\cup_{k\text{ odd}}\partial J_{k}$. Since any boundary point of an even block is the interior
of some odd block, it follows from \autoref{lem:mask} that
$\mathcal{W}r=\mathcal{W}(Mr)$ for any vector $r$.\footnote{This is most easily seen by noting that the PAR sampler is equivalent to a two-block sampler, comprising the composite blocks: $J_{\text{odd}} \eqdef \bigcup_{k \text{ odd}} J_k$ and $J_{\text{even}} \eqdef \bigcup_{k \text{ even}} J_k$. The statement
follows by applying \autoref{lem:mask} to these two blocks.}
%
%
%
To complete the proof we thus need to bound $\| M\mathcal{W}\|_\infty \leq\lambda^{2}$.
Let $i \in \partial_{\text{odd}}$ and let $J_k$ be the block such that $i \in J_k$. Note that $k$ is even and $J_k$ is
the only block containing~$i$. Thus,
\(
  [\mathcal{W}\mathbf{1}]_{i}=[\mathcal{W}_{even}\mathcal{W}_{odd}\mathbf{1}]_{i}=[W^{J_{k}}\mathcal{W}_{odd}\mathbf{1}]_{i} .
\)
Now assume $[\mathcal{W}_{odd}\mathbf{1}]_{i}\leq\lambda$
for all $i$ which are boundaries of even blocks (a fact we will prove next). Thus
\[
[W^{J_{k}}\mathcal{W}_{odd}\mathbf{1}]_{i}=W_{i,\partial_{-}J_{k}}^{J_{k}}[\mathcal{W}_{odd}\mathbf{1}]_{\partial_{-}J_{k}}+W_{i,\partial_{+}J_{k}}^{J_{k}}[W^{J_{k}}\mathcal{W}_{odd}\mathbf{1}]_{\partial_{+}J_{k}} \leq \lambda^{2}.
\]
To conclude: let $i$ be a boundary of an even block and let $i\in J_k$, the only odd block containing~$i$. Then $[\mathcal{W}_{odd}\mathbf{1}]_{i}=[W^{J_k}\mathbf{1}]_{i}=W^{J_k}_{i,\partial_{-}J_k}+W^{J_k}_{i,\partial_{+}J_k} \leq \lambda$.
\hfill $\qedsymbol$




\subsection{Proof of  \autoref{thm:parallelPgibbsGeneral}}
We prove the following more general version (while  \autoref{thm:parallelPgibbsGeneral} was stated for $r=\mathbf{1}$).
\begin{lemma} \label{lem:parallelPgibbs}
Let vector $r$ satisfy
\[
[r]_{i} =
\begin{cases}
  a, & i\in J_{-k}^{c}, \, k \text{ even}, \\
  a', & i\in J_{-k}^{c}, \, k \text{ odd}, \\
  b, & i\in J_{k}\cap J_{-k}, \, k \text{ even}, 
\end{cases}
\]
for some positive constants $a,a',b$. Then 
\[
[\widehat{\W}r]_{i}\leq
\begin{cases}
  \rate^{2}a+\epsilon\left(\rate c+2\rate a+2\epsilon c+La(1\vee\maxsumW)+\epsilon Lc\right) , 
  & i\in J_{-k}^{c}, \, k \text{ even}, \\
  \rate a+\epsilon\left(2a+L(a'\vee b)\right) ,
  & i\in J_{-k}^{c}, \, k \text{ odd}, \\
  \rate\maxsumW a+\epsilon\left(\maxsumW c+2\rate a+2\epsilon c+La(1\vee\maxsumW)+\epsilon Lc\right),
  & i\in J_{k}\cap J_{-k}, \, k \text{ even}, 
\end{cases}
\]
where $c=2a+L(a'\vee b)$ and the remaining constants were defined in \autoref{thm:parallelPgibbsGeneral}.
\end{lemma}

\begin{proof}
Note that $\widehat{\W}=\widehat{\W}_{\text{even}}\widehat{\W}_{\text{odd}}$ where
\begin{align*}
\widehat{\W}_{\text{odd}}&=\widehat{W}^{J_{m}}\cdots\widehat{W}^{J_{3}}\widehat{W}^{J_{1}} &
&\text{and} & \widehat{\W}_{\text{even}}&=\widehat{W}^{J_{m-1}}\cdots\widehat{W}^{J_{4}}\widehat{W}^{J_{2}} \eqsp.
\end{align*}
Let $r'=\widehat{\W}_{\text{odd}}r$ and $r''=\widehat{\W}_{\text{even}}r'$.
Recall that for any vector $s$, $\widehat{W}^{J_{k}}s$ differs from $s$
only in components indexed by $J_{k}$.
Thus, by \hypref{asmp:algs:ordered_blocks},~\hypref{asmp:algs:overlap},
 $\widehat{\W}_{\text{odd}}r$
can be studied by considering each term $\widehat{W}^{J_{k}}r$, $k$
odd, separately.  We have for $k$ odd and $i\in J_k$
\begin{align*}
[\widehat{W}^{J_{k}}r]_{i} & =(W_{i,\partial_{-}J_{k}}^{J_{k}}+\epsilon)r_{\partial_{-}J_{k}}+(W_{i,\partial_{+}J_{k}}^{J_{k}}+\epsilon)r_{\partial_{+}J_{k}}+\epsilon\sum_{j\in J_{k}}r_{j}\\
 & =(W_{i,\partial_{-}J_{k}}^{J_{k}}+W_{i,\partial_{+}J_{k}}^{J_{k}}+2\epsilon)a+\epsilon\sum_{j\in J_{k}}r_{j}\\
 & \leq(W_{i,\partial_{-}J_{k}}^{J_{k}}+W_{i,\partial_{+}J_{k}}^{J_{k}}+2\epsilon)a+\epsilon|J_{k}|(a'\vee b) \eqsp.
\end{align*}
The second line follows by the assumption on $r$ and the fact that
boundaries of an odd numbered block lie in the adjacent even blocks.
For $i\in J_{-k}^{c}$, $W_{i,\partial_{-}J_{k}}^{J_{k}}+W_{i,\partial_{+}J_{k}}^{J_{k}}\leq\rate$
and since $|J_{k}|\leq L$, 
\[
[\widehat{W}^{J_{k}}r]_{i}\leq\rate a+\epsilon\left(2a+L(a'\vee b)\right)  \eqsp.
\]
For $i\in J_{k}\cap J_{-k}$, $W_{i,\partial_{-}J_{k}}^{J_{k}}+W_{i,\partial_{+}J_{k}}^{J_{k}}\leq\maxsumW$
and thus
\[
[\widehat{W}^{J_{k}}r]_{i}\leq\maxsumW a+\epsilon\left(2a+L(a'\vee b)\right)  \eqsp.
\]
It follows that, $r'=\widehat{\W}_{\text{odd}}r$, 
\[
[r']_{i}\leq
\begin{cases}
  a, & i\in J_{-k}^{c}, \, k \text{ even}, \\
  \rate a+\epsilon\left(2a+L(a'\vee b)\right),  & i\in J_{-k}^{c}, \, k \text{ odd}, \\
  \maxsumW a+\epsilon\left(2a+L(a'\vee b)\right), & i\in J_{k}\cap J_{-k}, \, k \text{ even},
\end{cases}
\]

Similarly, $r''=\widehat{\W}_{\text{even}}r'$
can be studied by considering each term $\widehat{W}^{J_{k}}r'$, $k$
even, separately.
 Let $c=2a+L(a'\vee b)$.
For $k$ even and $i\in J_k$, 
\begin{align*}
  [\widehat{W}^{J_{k}}r']_{i} &
  =(W_{i,\partial_{-}J_{k}}^{J_{k}}+\epsilon)r'_{\partial_{-}J_{k}}
  +(W_{i,\partial_{+}J_{k}}^{J_{k}}+\epsilon)r'_{\partial_{+}J_{k}}+\epsilon\sum_{j\in J_{k}}r'_{j}\\
 & \leq(W_{i,\partial_{-}J_{k}}^{J_{k}}+W_{i,\partial_{+}J_{k}}^{J_{k}}+2\epsilon)
 (\rate a+\epsilon c)+\epsilon\sum_{j\in J_{k}} r'_{j}\\
 & \leq(W_{i,\partial_{-}J_{k}}^{J_{k}}+W_{i,\partial_{+}J_{k}}^{J_{k}}+2\epsilon)
 (\rate a+\epsilon c)+\epsilon|J_{k}|\left(a(1\vee\maxsumW)+\epsilon c\right) \eqsp.
\end{align*}
For $i\in  J_{-k}^{c}$, $W_{i,\partial_{-}J_{k}}^{J_{k}}+W_{i,\partial_{+}J_{k}}^{J_{k}}\leq\rate$
and thus,
\begin{align*}
[\widehat{W}^{J_{k}}r']_{i} & \leq\rate^{2}a+\epsilon\left(\rate c+2\rate a+2\epsilon c+La(1\vee\maxsumW)+\epsilon Lc\right) \eqsp.
\end{align*}
For $i\in J_{k}\cap J_{-k}$, $W_{i,\partial_{-}J_{k}}^{J_{k}}+W_{i,\partial_{+}J_{k}}^{J_{k}}\leq\maxsumW$
and thus,
\begin{align*}
[\widehat{W}^{J_{k}}r']_{i} & \leq\rate\maxsumW a+\epsilon\left(\maxsumW c+2\rate a+2\epsilon c+La(1\vee\maxsumW)+\epsilon Lc\right) \eqsp.
\end{align*}
In summary, $r''=\widehat{\W}_{\text{even}}r'$,
\[
[r'']_{i}\leq
\begin{cases}
  \rate^{2}a+\epsilon\left(\rate c+2\rate a+2\epsilon c+La(1\vee\maxsumW)+\epsilon Lc\right),
  & i\in J_{-k}^{c}, \, k \text{ even}, \\
  \rate a+\epsilon\left(2a+L(a'\vee b)\right),
  & i\in J_{-k}^{c}, \, k \text{ odd}, \\
  \rate\maxsumW a+\epsilon\left(\maxsumW c+2\rate a+2\epsilon c+La(1\vee\maxsumW)+\epsilon Lc\right),
  & i\in J_{k}\cap J_{-k}, \, k \text{ even}.
\end{cases}
\]

\end{proof}

\subsection{Proof of \autoref{thm:sequentialPgibbsGeneral}}
The proof is inductive. For any vector $r$, observe that $\widehat{W}^{J}r$
differs from $r$ only in components indexed by $J$. In particular,
for any $i \in J_1$, 
$[\widehat{W}^{J_{1}}\mathbf{1}]_{i}=W_{i,\partial_{+}J_{1}}^{J_{1}}+(|J_{1}|+1)\epsilon$.
Thus,
\begin{align*}
  [\widehat{W}^{J_{1}}\mathbf{1}]_{i}\leq
  \begin{cases}
    \rate+c\epsilon, & \text{for $i\in J_{-1}^{c}$}, \\
    \maxsumW+c\epsilon,  & \text{for $i\in J_{1}\cap J_{2}$}.
  \end{cases}
\end{align*}
We separately bound the terms for $i\in  J_{-1}^{c}$ and
$i\in J_{1}\cap J_{2}$ since
$W_{i,\partial_{+} J_{1}}^{J_{1}}$ can approach 1 as $i$ approaches the boundary
$\partial_{+} J_{1}$ (recall $\widehat{W}_{i,\partial_{+} J_{1}}^{J_{1}}=W_{i,\partial_{+} J_{1}}^{J_{1}}+\epsilon$ for $i \in J_1$).

Now assume 
\begin{align*}
[\widehat{W}^{J_{k-1}}\cdots\widehat{W}^{J_{1}}\mathbf{1}]_{i}\leq
\begin{cases}
  \rate+c\epsilon ,
  & \text{for  $i\in  J_{-(k-1)}^{c}$ (middle of block)}, \\ 
  \ratetwo+2c\epsilon  ,
 & \text{for $i\in J_{k-2}\cap J_{k-1}$ (left overlap)}, \\
 \maxsumW(\rate\vee1)+c\epsilon ,
 & \text{for $i\in J_{k-1}\cap J_{k}$ (right overlap)}.
\end{cases}
\end{align*}
We have
\begin{align}
\nonumber
[\widehat{W}^{J_{k}}\widehat{W}^{J_{k-1}}\cdots\widehat{W}^{J_{1}}\mathbf{1}]_{i} 
&=  \sum_{j\in J_{k}^{+}}\widehat{W}_{i,j}^{J_{k}}[\widehat{W}^{J_{k-1}}\cdots\widehat{W}^{J_{1}}\mathbf{1}]_{j} \\
\label{eq:proof-refined-lr-1}
&=  \sum_{j\in\partial J_{k}}(W_{i,j}^{J_{k}}+\epsilon)[\widehat{W}^{J_{k-1}}\cdots\widehat{W}^{J_{1}}\mathbf{1}]_{j}
  +\sum_{j\in J_{k}}\epsilon[\widehat{W}^{J_{k-1}}\cdots\widehat{W}^{J_{1}}\mathbf{1}]_{j} \eqsp.
\end{align}
The first term of the sum can be simplified to
\begin{align*}
\sum_{j\in\partial J_{k}}(W_{i,j}^{J_{k}}+{}&\epsilon)[\widehat{W}^{J_{k-1}}\cdots\widehat{W}^{J_{1}}\mathbf{1}]_{j}  \\
&= (W_{i,\partial_{-}J_{k}}^{J_{k}}+\epsilon)[\widehat{W}^{J_{k-1}}\cdots\widehat{W}^{J_{1}}\mathbf{1}]_{\partial_{-}J_{k}}
+(W_{i,\partial_{+}J_{k}}^{J_{k}}+\epsilon)[\widehat{W}^{J_{k-1}}\cdots\widehat{W}^{J_{1}}\mathbf{1}]_{\partial_{+}J_{k}} \\
&= (W_{i,\partial_{-}J_{k}}^{J_{k}}+\epsilon)(\rate+c\epsilon)+(W_{i,\partial_{+}J_{k}}^{J_{k}}+\epsilon) \\
&= W_{i,\partial_{-}J_{k}}^{J_{k}}\rate+W_{i,\partial_{+}J_{k}}^{J_{k}}+\epsilon(W_{i,\partial_{-}J_{k}}^{J_{k}}c+\rate+c\epsilon+1) \eqsp.
\end{align*}
The second term of \eqref{eq:proof-refined-lr-1} is
\begin{align*}
\nonumber
\sum_{j\in J_{k}}\epsilon[\widehat{W}^{J_{k-1}}\cdots\widehat{W}^{J_{1}}\mathbf{1}]_{j} &= \sum_{j\in J_{k-1}\cap J_{k}}\epsilon[\widehat{W}^{J_{k-1}}\cdots\widehat{W}^{J_{1}}\mathbf{1}]_{j}+\sum_{j\in J_{k-1}^{c}\cap J_{k}}\epsilon\\
\nonumber
 & \le |J_{k-1}\cap J_{k}|\epsilon(\maxsumW(\rate\vee1)+c\epsilon)+\epsilon|J_{k-1}^{c}\cap J_{k}|\\
 & \le \epsilon(\maxsumW(\rate\vee1)\vee1)|J_{k}|+c\epsilon^{2}|J_{k-1}\cap J_{k}| \eqsp.
\end{align*}
Thus 
\begin{align*}
[\widehat{W}^{J_{k}}\widehat{W}^{J_{k-1}}\cdots\widehat{W}^{J_{1}}\mathbf{1}]_{i} \leq {}&W_{i,\partial_{-}J_{k}}^{J_{k}}\rate+W_{i,\partial_{+}J_{k}}^{J_{k}}\\
&+\epsilon(W_{i,\partial_{-}J_{k}}^{J_{k}}c+\rate+c\epsilon+1
+(\maxsumW(\rate\vee1)\vee1)|J_{k}|+c\epsilon|J_{k-1}\cap J_{k}|) \eqsp.
\end{align*}
For $i\in J_{k}\cap J_{k+1}$, $W_{i,\partial_{-}J_{k}}^{J_{k}}\rate+W_{i,\partial_{+}J_{k}}^{J_{k}}\leq\maxsumW(\rate\vee1)$.
For $i\in  J_{-k}^{c},$ $W_{i,\partial_{-}J_{k}}^{J_{k}}\rate+W_{i,\partial_{+}J_{k}}^{J_{k}}\leq\rate$
by the definition of $\rate$. For $i\in J_{k-1}\cap J_{k},$
$W_{i,\partial_{-}J_{k}}^{J_{k}}\rate+W_{i,\partial_{+}J_{k}}^{J_{k}}\leq\ratetwo$
by the definition of $\ratetwo$. 

For $i\in J_{k}\cap J_{k-1}^{c},$ the coefficient of $\epsilon$
is by itself bounded by $c$ since
\[
\frac{\rate+1+(\maxsumW(\rate\vee1)\vee1)|J_{k}|}{1-W_{i,\partial_{-}J_{k}}^{J_{k}}-\epsilon-\epsilon|J_{k-1}\cap J_{k}|}\leq c \eqsp.
\]
For $i\in J_{k-1}\cap J_{k}$, the coefficient of $\epsilon$ is
\begin{align*}
 W_{i,\partial_{-}J_{k}}^{J_{k}}c
 {}&+\rate+c\epsilon+1+(\maxsumW(\rate\vee1)\vee1)|J_{k}|+c\epsilon|J_{k-1}\cap J_{k}|\\
 & \leq c+\rate+c\epsilon+1+(\maxsumW(\rate\vee1)\vee1)|J_{k}|+c\epsilon|J_{k-1}\cap J_{k}|)
  \leq2c \eqsp.
\end{align*}
\hfill $\qedsymbol$


\section{Additional Proofs} \label{app:proofs:additional}
\subsection{Proof of \autoref{prop:coupling-HMM}}

\begin{lemma}
\label{lem:h-coupling}
For each $i\in{1,\ldots,t}$, let $P_{i}(x_{i-1},\rmd x_{i})$ be a Markov
transition kernel on $\Xset$. For some integer $h>0$, assume that the
composite transition kernels $Q_{i}(x_{(i-1)h},\rmd x_{ih})=(P_{(i-1)h+1}\ldots P_{ih})(x_{(i-1)h},\rmd x_{ih})$
satisfy the following minorisation condition: there exists probability
measures $\nu_{i}(\rmd x)$ on $\Xset$ and a common constant $\alpha \in [0,1]$
such that $Q_{i}\geq\alpha\nu_{i}$. Then for the probability measures
(on the product space $\Xset^{t}$) $\prod_{i=1}^{t}P(x_{i-1}, \rmd x_{i})$
and $\prod_{i=1}^{t}P(y_{i-1},\rmd y_{i})$, and for any $x_{0}$ and $y_{0}$,
there exists a coupling $\Psi$ such that if $(X_{1:t},Y_{1:t})\sim\Psi$
then $\Pr(X_{i}\neq Y_{i})\leq(1-\alpha)^{\lfloor i/h\rfloor}$.
\end{lemma}
\begin{proof}
When $h=1$ the results follows standard arguments \citep{Lindvall:2002}
and is repeated here for the sake of completeness. 
The coupling $\Psi$ attempts to couple $(X_{1},Y_{1})$, followed
by $(X_{2},Y_{2})$ etc. Specifically, if $X_{i-1}=Y_{i-1},$ then
draw $X_{i}$ from $P_i(X_{i-1},\rmd x_{i})$ and set $Y_{i}=X_{i}$. If
$X_{i-1}\neq Y_{i-1},$ draw $(X_{i},Y_{i})$ from the measure
\[
\alpha \nu_i(\rmd x_{i})\delta_{x_{i}}(\rmd y_{i})+\left(1-\alpha\right)^{-1}\left(P_{i}(X_{i-1},\rmd x_{i})-\alpha \nu(\rmd x_{i})\right)\left(P_{i}(Y_{i-1},\rmd y_{i})-\alpha \nu(\rmd y_{i})\right)
\]
where $\delta_x$ is the atom measure.
It now follows that $\Pr(X_{i}\neq Y_{i})\leq(1-\alpha)\Pr(X_{i-1}\neq Y_{i-1})\leq(1-\alpha)^{i}.$

For $h>1$ we couple the skeleton process $(X_{h},Y_{h}),(X_{2h},Y_{2h}),\ldots$
using a similar scheme as for $h=1$. (Note that the $h$-skeleton
$X$-system has transition kernels $Q_{i}$ that satisfy a minorisation
condition which is to be used in the same manner as in the proof for $h=1$.) Let $i$ be the first instance that $X_{ih}=Y_{ih}$,
i.e. $X_{jh}\neq Y_{jh}$ for $j<i$. Then, simulate the future $X$-process,
$X_{k}$ for $k=ih+1,ih+2,\ldots$ from $P_{k}$ and set $Y_{k}=X_{k}$.
The non-skeleton terms $X_{k}$ and $Y_{k}$ for $k<ih$ are simulated
independently from their respective conditional laws.
Thus for $i\geq h$, $\Pr(X_{i}\neq Y_{i})\leq\Pr(X_{ h \lfloor i/h\rfloor}\neq Y_{h \lfloor i/h\rfloor})$
which is bounded above by $(1-\alpha)^{\lfloor i/h\rfloor}$.
\end{proof}

\begin{proof}(\autoref{prop:coupling-HMM})
For notational brevity we write $\rmd x$ instead of $\nu(\rmd x)$, where $\nu$ is the dominating probability measure
defined in \hypref{asmp:strong-mixing-m}.

Recall that under \hypref{asmp:algs:ordered_blocks}, $J$ is an interval and we can write $J = \crange{\si}{\ui}$.
To compute $W_{i,j}^J$ for $i \in J$ we use a coupling as in \eqref{eq:BCW:WJdefn}:
\begin{align} \label{eq:app:prop2:Wij}
  W_{i,j}^{J} \eqdef
  \!\!
  \sup_{
    {\scriptsize \begin{array}{cc}
        x,z\in \Xset^{\T} \\
        x_{-j}=z_{-j} 
      \end{array}}} 
  \!\!
  \Cpl_{j,x,z}^{J}(X'_i \neq Z'_i),
\end{align}
where $\Cpl_{j,x,z}^{J}$ is a coupling of $\phi_x^J$ and $\phi_z^J$ for $x_{-j} = z_{-j}$. We know from Lemma~\ref{lem:structure-W-general} that we only need to consider the cases $j = \si-1$ and $j = \ui+1$.
Consider first $j = \si-1$ (assuming $\si>1$).

Write the density of $\phi_{x}^{J}$ as $p(x_{s},\ldots,x_{u}|x_{s-1},x_{u+1},y_{s:u})=\prod_{i=s}^{u}p(x_{i}|x_{i-1},x_{u+1},y_{i:u})$,
which is a product of inhomogeneous Markov transition kernels. (A
similar expression follows if written backwards, i.e. $\prod_{i=s}^{u}p(x_{i}|x_{i+1},x_{s-1},y_{s:i})$.)

We show that the composite kernel formed by any $h$ (for $h$ defined
in \hypref{asmp:strong-mixing-m} and \hypref{asmp:strong-mixing-g}) successive kernels of the given inhomogeneous product,
i.e. \( \int\prod_{i=t+1}^{t+h}p(x_{i}|x_{i-1},x_{u+1},y_{i:u}) \rmd x_{t+1:t+h-1}, \)
satisfies a minorisation condition with respect to some probability
measure and the time-uniform constant $\delta^{\frac{1-h}{h}}\frac{\sigma_{-}}{\sigma_{+}}.$
Thus the coupling $\Psi_{j,x,z}^{J}$ can be defined as in \autoref{lem:h-coupling} to complete the proof.

Let $t \leq \ui-\h$ be some time index and consider
\begin{align}
  \label{eq:app:prop2:xth}
  p(x_{t+\h} \mid x_t, x_{\ui+1}, y_{t+1:\ui}) = \frac{ p(x_{\ui+1}, y_{t+\h:\ui} \mid x_{t+\h}) p(x_{t+\h} \mid x_t, y_{t+1:t+\h-1}) }{ \int  p(x_{\ui+1}, y_{t+\h:\ui} \mid x_{t+\h}) p(x_{t+\h} \mid x_t, y_{t+1:t+\h-1}) \rmd x_{t+\h} }.
\end{align}
We have
\begin{align*}
  p(x_{t+\h} \mid x_t, y_{t+1:t+\h-1}) = \frac{ \int\prod_{j=t+1}^{t+\h-1} \left \{ g(x_{j} , y_{j}) m(x_{j-1}, x_{j}) \right\} m(x_{t+\h-1}, x_{t+\h} ) \rmd x_{t+1:t+\h-1} }
    {\int\prod_{j=t+1}^{t+\h-1} \left\{ g(x_{j}, y_{j}) m(x_{j-1}, x_{j}) \right\} m(x_{t+\h-1}, x_{t+\h} ) \rmd x_{t+1:t+\h} },
\end{align*}
which implies, using \hypref{asmp:strong-mixing-m} and \hypref{asmp:strong-mixing-g},
\begin{align*}
  \delta^{\frac{ 1-\h}{\h} }\sigma_{-} \leq p(x_{t+\h} \mid x_t, y_{t+1:t+\h-1}) \leq \sigma_{+}.
\end{align*}
Plugging these two bounds into the numerator and denominator, respectively, of \eqref{eq:app:prop2:xth} we get
\begin{align}
  \label{eq:app:prop2:xth-bound}
  p(x_{t+\h} \mid x_t, x_{\ui+1}, y_{t+1:\ui}) \geq
  \delta^{\frac{ 1-\h}{\h} }\frac{ \sigma_{-}}{\sigma_{+} }
  \frac{ p(x_{\ui+1}, y_{t+\h:\ui} \mid x_{t+\h})  }{ \int  p(x_{\ui+1}, y_{t+\h:\ui} \mid x_{t+\h}) \rmd x_{t+\h}}.
\end{align}
The proof for $j = \ui+1$ follows analogously by considering a backward decomposition of
$p(x_{i} \mid x_{\si-1}, x_{\ui+1}, y_{\si:\ui})$ and minorising $p(x_{t-h} \mid x_t, x_{\si-1}, y_{\si:t-1})$.
\end{proof}

\subsection{Proof of \autoref{lem:pgibbs-ergodicity}}
\subsubsection{Part 1 -- Uniform minorisation of the blocked \pg kernel}
We consider the blocked \pg kernel for block $J = \crange{\si}{\ui}$, $Q_N^J$.
The following proof is a slight modification of the proof of Proposition~5 by \citet{LindstenDM:2015},
which amount to taking the conditioning on the end-point $x_{\ui+1}$ into account under
the $\h$-step forgetting condition \hypref{asmp:strong-mixing-m-minus}.
It is known from \citet[Theorem~1]{LindstenDM:2015} that, for any $x_{\extended{J}} \in \Xset^{\card{J}+2}$
and measurable $A\subset \Xset^{\card{J}}$,
\(
  \kernel{Q_N^J}{x_{\extended{J}}}{A} \geq (1-\epsilon(\Np, \card{J})) \phi_x^J(A),
\)
where
\begin{align*}
  \epsilon(\Np, \card{J}) &= 1-\prod_{t=\si}^\ui \frac{N-1}{2B_t^J + N -2}, \\
  B_t^J &= \left\{ \sup_{x_t} \frac{ p( y_{t:\ui}, x_{\ui+1} \mid x_t)}{p(y_{t:\ui}, x_{\ui+1} \mid x_{\si-1}, y_{\si:t-1})} \right\}
  \vee  \max_{0\leq\ell < \ui-t} \left\{ \sup_{x_t} \frac{ p( y_{t:t+\ell} \mid x_t)}{p(y_{t:t+\ell} \mid  x_{\si-1}, y_{\si:t-1})} \right\}.
\end{align*}
Assume \hypref{asmp:strong-mixing-m},\hypref{asmp:strong-mixing-m}.
From \citet[Proposition~5]{LindstenDM:2015} if follows that the second term in the definition of $B_t^J$
is bounded by $\delta \frac{\sigma_{+}}{\sigma_{-}}$. It remains to bound the first term, incorporating the
dependence on the boundary point $x_{\ui+1}$.

Consider first $t \leq \ui-\h+1$. We have for the numerator
\begin{multline*}
  p( y_{t:\ui}, x_{\ui+1} \mid x_t) 
  = \int \prod_{j=t}^{t+\h-1} \left(g(x_j, y_j) m(x_j, x_{j+1})\right)
  \prod_{j=t+\h}^{\ui} \left(g(x_j, y_j) m(x_j, x_{j+1})\right) \rmd x_{t+1:\ui} \\
  \leq \prod_{j=t}^{t+\h-1} \left\{ \sup_{x} g(x,y_j) \right\}
  \int \prod_{j=t}^{t+\h-1} m(x_j, x_{j+1})
  \prod_{j=t+\h}^{\ui} \left(g(x_j, y_j) m(x_j, x_{j+1})\right) \rmd x_{t+1:\ui} \\
  \leq \sigma_{+} \prod_{j=t}^{t+\h-1} \left\{ \sup_{x} g(x,y_j) \right\}
  \int  \prod_{j=t+\h}^{\ui} \left(g(x_j, y_j) m(x_j, x_{j+1})\right) \rmd x_{t+\h:\ui} \eqsp.
\end{multline*}
Analogously we get for the denominator
\begin{multline*}
  p(y_{t:\ui}, x_{\ui+1} \mid x_{\si-1}, y_{\si:t-1}) 
  = \int p(y_{t:\ui}, x_{\ui+1} \mid x_t) p(x_t \mid x_{\si-1}, y_{\si:t-1}) \rmd x_t \\
  \geq 
  \sigma_{-} \prod_{j=t}^{t+\h-1} \left\{ \inf_{x} g(x,y_j) \right\}
  \int \prod_{j=t+\h}^{\ui} \left(g(x_j, y_j) m(x_j, x_{j+1})\right) \rmd x_{t+\h:\ui} \eqsp.
\end{multline*}
It follows that for $t \leq \ui-\h+1$,
\begin{align*}
  \sup_{x_t} \frac{ p( y_{t:\ui}, x_{\ui+1} \mid x_t) }{ p(y_{t:\ui}, x_{\ui+1} \mid x_{\si-1}, y_{\si:t-1} ) }  \leq \delta \frac{\sigma_{+}}{\sigma_{-}}
\end{align*}
Next, consider the case $t > \ui-\h+1$. We have for the numerator
\begin{multline*}
  p( y_{t:\ui}, x_{\ui+1} \mid x_t) = \int \prod_{j=t}^\ui \left( g(x_j, y_j) m(x_j, x_{j+1}) \right) \rmd x_{t+1:\ui} \\
  \leq \prod_{j=t}^{\ui} \left\{ \sup_{x} g(x,y_j) \right\}
  \int \prod_{j=t}^{\ui} m(x_j, x_{j+1}) \rmd x_{t+1:\ui}
  \leq \sigma_{+} \prod_{j=t}^{\ui} \left\{ \sup_{x} g(x,y_j) \right\} \eqsp.
\end{multline*}
For the denominator we write
\begin{align*}
  p(y_{t:\ui}, x_{\ui+1} \mid x_{\si-1}, y_{\si:t-1}) 
  = \frac{ p(y_{\ui-\h+1:\ui}, x_{\ui+1} \mid x_{\si-1}, y_{\si:\ui-\h}) }{ p(y_{\ui-\h+1:t-1} \mid x_{\si-1}, y_{\si:\ui-\h})  } \eqsp.
\end{align*}
We have
\begin{multline*}
  p(y_{\ui-\h+1:\ui}, x_{\ui+1} \mid  x_{\si-1}, y_{\si:\ui-\h}) \\
  = \int \prod_{j=\ui-\h+1}^\ui \left( g(x_j, y_j) m(x_{j}, x_{j+1}) \right) p(x_{\ui-\h+1} \mid  x_{\si-1},  y_{\si:\ui-\h}) \rmd x_{\ui-\h+1:\ui} \\
  \geq \prod_{j=\ui-\h+1}^{\ui} \left\{ \inf_{x} g(x,y_j) \right\}
  \int \prod_{j=\ui-\h+1}^\ui  m(x_{j}, x_{j+1}) p(x_{\ui-\h+1} \mid  x_{\si-1}, y_{\si:\ui-\h}) \rmd x_{\ui-\h+1:\ui} \\
  \geq \sigma_{-} \prod_{j=\ui-\h+1}^{\ui} \left\{ \inf_{x} g(x,y_j) \right\} \eqsp,
\end{multline*}
and
\begin{multline*}
  p(y_{\ui-\h+1:t-1} \mid x_{\si-1}, y_{\si:\ui-\h}) \\=
  \int \prod_{j=\ui-\h+1}^{t-1} \left( g(x_j, y_j) m(x_{j}, x_{j+1}) \right) p(x_{\ui-\h+1} \mid x_{\si-1}, y_{\si:\ui-\h}) \rmd x_{\ui-\h+1:t} \\
  \leq \sigma_{+} \prod_{j=\ui-\h+1}^{t-1} \left\{ \sup_{x} g(x,y_j) \right\} \eqsp.
\end{multline*}
In summary we get
\begin{align*}
  \sup_{x_t} \frac{ p( y_{t:\ui}, x_{\ui+1} \mid x_t) }{ p(y_{t:\ui}, x_{\ui+1} \mid x_{\si-1}, y_{\si:t-1} ) }  \leq \delta \frac{\sigma_{+}}{\sigma_{-}}
\end{align*}
also for $t > \ui-\h+1$, and thus $B_t^J \leq \delta \frac{\sigma_{+}}{\sigma_{-}}$ for all $t \in J$. The result follows.
\hfill$\qedsymbol$

\subsubsection{Part 2 -- Wasserstein estimate for the blocked \pg kernel}
The statement is a corollary to the following lemma.

\begin{lemma}
  Let $P$ and $Q$ be two Markov kernels on $\Xset^\T$ and assume that there exists a constant $\epsilon \in [0,1]$ such that
  $Q(x,\rmd y) \geq (1-\epsilon) P(x, \rmd y)$ for all $x \in \Xset^\T$. 
  For $x,y\in\Xset^\T$ such that $x_{-j}=y_{-j}$, let $\Cpl_{j,x,y}$ be a coupling of $P(x, \cdot)$ and $P(y, \cdot)$ and let $W$ be the matrix (a Wasserstein matrix for~$P$) 
\[
    W_{i,j} =   \sup_{ \scriptsize \begin{array}{c} x, y \in \Xset^\T \\ x_{-j}=y_{-j}  \end{array}}  \Cpl_{j,x,y}(X'_i \neq Y'_i) 
    =  \sup_{ \scriptsize \begin{array}{c} x, y \in \Xset^\T \\ x_{-j}=y_{-j}  \end{array}} \int \Cpl_{j,x,y}(\rmd x', \rmd y') \I_{[x'_i \neq y'_i]}.
\]
Then, the matrix $\widehat W$ with $\widehat W_{i,j} = W_{i,j} + \epsilon$ is a Wasserstein matrix for $Q$.
\end{lemma}
\begin{proof}
  For $\epsilon = 0$ the result is immediate. Hence, consider $\epsilon > 0$.
  For any $x,y \in \Xset^\T$, define $r_{x} = \frac{1}{\epsilon}(Q(x,\cdot) - (1-\epsilon)P(x,\cdot))$
  and let $R_{x,y}$ be a coupling of $r_{x}$ and $r_{y}$. It follows that, for any $x,y \in \Xset^\T$
  with $x_{-j} = y_{-j}$,  $(1-\epsilon) \Psi_{j,x,y} + \epsilon R_{x,y}$ is a coupling
  of $Q(x,\cdot)$ and $Q(y, \cdot)$. We have,
  \begin{align*}
    (1-\epsilon) \Psi_{j,x,y}(X'_i \neq Y'_i) + \epsilon R_{x,y}(X'_i \neq Y'_i) \leq (1-\epsilon) W_{i,j} + \epsilon \leq \widehat W_{i,j} \eqsp.
  \end{align*}
\end{proof}


\small
\bibliographystyle{chicago}
\bibliography{references}

\end{document}

%% file: fig-blocking-scheme.tex
\tikzstyle{block}=[draw, thick, minimum width=3cm, minimum height = 0.2cm, align=center]

\begin{tikzpicture}[node distance=4cm,auto]
  \node[block] (J1) at (0,0) {$J_1$};
  \node[block,right=.8cm of J1] (J3) {$J_3$};
  \node[block,right=.8cm of J3] (J5) {$J_5$};
  \node[block] (J2) at (2,-0.62) {$J_2$};
  \node[block,right=.8cm of J2] (J4) {$J_4$};
  \draw[thick] (-1.5,-1)--(-1.5,-1.3);
  \draw[thick] (-1.2,-1)--(-1.2,-1.3);
  \draw[thick] (8.8,-1)--(8.8,-1.3);
  \draw[thick] (9.1,-1)--(9.1,-1.3);
  \draw[thick] (-1.5,-1.15)--(9.1,-1.15);
  \node[anchor=west] at (-1.55,-1.6) {$1\; \cdots$};
  \node[anchor=east] at (9.16,-1.6) {$\cdots\; n$};
\end{tikzpicture}

%% file: blocking-pg-v1-final.bbl
\begin{thebibliography}{}

\bibitem[\protect\citeauthoryear{Andrieu, Doucet, and Holenstein}{Andrieu
  et~al.}{2010}]{AndrieuDH:2010}
Andrieu, C., A.~Doucet, and R.~Holenstein (2010).
\newblock Particle {M}arkov chain {M}onte {C}arlo methods.
\newblock {\em Journal of the Royal Statistical Society: Series B\/}~{\em
  72\/}(3), 269--342.

\bibitem[\protect\citeauthoryear{Andrieu, Lee, and Vihola}{Andrieu
  et~al.}{2015}]{AndrieuLV:2015}
Andrieu, C., A.~Lee, and M.~Vihola (2015, April).
\newblock Uniform ergodicity of the iterated conditional {SMC} and geometric
  ergodicity of particle {G}ibbs samplers.
\newblock arXiv.org, arXiv:1312.6432v2.

\bibitem[\protect\citeauthoryear{Capp\'e, Moulines, and Ryd\'en}{Capp\'e
  et~al.}{2005}]{CappeMR:2005}
Capp\'e, O., E.~Moulines, and T.~Ryd\'en (2005).
\newblock {\em Inference in Hidden {M}arkov Models}.
\newblock Springer.

\bibitem[\protect\citeauthoryear{Carter and Kohn}{Carter and
  Kohn}{1994}]{CarterK:1994}
Carter, C.~K. and R.~Kohn (1994).
\newblock On {G}ibbs sampling for state space models.
\newblock {\em Biometrika\/}~{\em 81\/}(3), 541--553.

\bibitem[\protect\citeauthoryear{Chopin and Singh}{Chopin and
  Singh}{2015}]{ChopinS:2015}
Chopin, N. and S.~S. Singh (2015).
\newblock On particle {G}ibbs sampling.
\newblock {\em Bernoulli\/}~{\em 21\/}(3), 1855--1883.

\bibitem[\protect\citeauthoryear{Del~Moral}{Del~Moral}{2004}]{delmoral:2004}
Del~Moral, P. (2004).
\newblock {\em {F}eynman-{K}ac Formulae - Genealogical and Interacting Particle
  Systems with Applications}.
\newblock Probability and its Applications. Springer.

\bibitem[\protect\citeauthoryear{Doucet, Godsill, and Andrieu}{Doucet
  et~al.}{2000}]{DoucetGA:2000}
Doucet, A., S.~J. Godsill, and C.~Andrieu (2000).
\newblock On sequential {M}onte {C}arlo sampling methods for {B}ayesian
  filtering.
\newblock {\em Statistics and Computing\/}~{\em 10\/}(3), 197--208.

\bibitem[\protect\citeauthoryear{Doucet and Johansen}{Doucet and
  Johansen}{2011}]{DoucetJ:2011}
Doucet, A. and A.~Johansen (2011).
\newblock A tutorial on particle filtering and smoothing: Fifteen years later.
\newblock In D.~Crisan and B.~Rozovskii (Eds.), {\em The Oxford Handbook of
  Nonlinear Filtering}. Oxford University Press.

\bibitem[\protect\citeauthoryear{Follmer}{Follmer}{1982}]{Follmer:1982}
Follmer, H. (1982).
\newblock A covariance estimate for {G}ibbs measures.
\newblock {\em Journal of Functional Analysis\/}~{\em 46}, 387--395.

\bibitem[\protect\citeauthoryear{Fr\"uhwirth-Schnatter}{Fr\"uhwirth-Schnatter}{1994}]{Fruhwirth-Schnatter:1994}
Fr\"uhwirth-Schnatter, S. (1994).
\newblock Data augmentation and dynamic linear models.
\newblock {\em Journal of Time Series Analysis\/}~{\em 15\/}(2), 183--202.

\bibitem[\protect\citeauthoryear{Kuhlenschmidt}{Kuhlenschmidt}{2014}]{Kuhlenschmidt:2014}
Kuhlenschmidt, B. (2014).
\newblock {\em On the stability of Sequential Monte Carlo methods for parameter
  estimation}.
\newblock Ph.\ D. thesis, Cambridge Centre for Analysis, DPMMS, Cambridge
  University.

\bibitem[\protect\citeauthoryear{Lindsten, Douc, and Moulines}{Lindsten
  et~al.}{2015}]{LindstenDM:2015}
Lindsten, F., R.~Douc, and E.~Moulines (2015).
\newblock Uniform ergodicity of the particle {G}ibbs sampler.
\newblock {\em Scandinavian Journal of Statistics\/}~{\em 42\/}(3), 775--797.

\bibitem[\protect\citeauthoryear{Lindsten, Jordan, and Sch\"on}{Lindsten
  et~al.}{2014}]{LindstenJS:2014}
Lindsten, F., M.~I. Jordan, and T.~B. Sch\"on (2014).
\newblock Particle {G}ibbs with ancestor sampling.
\newblock {\em Journal of Machine Learning Research\/}~{\em 15}, 2145--2184.

\bibitem[\protect\citeauthoryear{Lindsten and Sch\"on}{Lindsten and
  Sch\"on}{2013}]{LindstenS:2013}
Lindsten, F. and T.~B. Sch\"on (2013).
\newblock Backward simulation methods for {M}onte {C}arlo statistical
  inference.
\newblock {\em Foundations and Trends in Machine Learning\/}~{\em 6\/}(1),
  1--143.

\bibitem[\protect\citeauthoryear{Lindvall}{Lindvall}{2002}]{Lindvall:2002}
Lindvall, T. (2002).
\newblock {\em Lectures on the Coupling Method}.
\newblock Dover.

\bibitem[\protect\citeauthoryear{Rebeschini and van Handel}{Rebeschini and van
  Handel}{2014}]{RebeschiniH:2014}
Rebeschini, P. and R.~van Handel (2014).
\newblock Comparison theorems for {G}ibbs measures.
\newblock {\em Journal of Statistical Physics\/}~{\em 157\/}(2), 234--281.

\bibitem[\protect\citeauthoryear{Wang and Wu}{Wang and Wu}{2014}]{WangW:2014}
Wang, N.-Y. and L.~Wu (2014).
\newblock Convergence rate and concentration inequalities for {G}ibbs sampling
  in high dimensions.
\newblock {\em Bernoulli\/}~{\em 20\/}(4), 1698--1716.

\bibitem[\protect\citeauthoryear{Whiteley}{Whiteley}{2010}]{Whiteley:2010}
Whiteley, N. (2010).
\newblock Discussion on {P}article {M}arkov chain {M}onte {C}arlo methods.
\newblock {\em Journal of the Royal Statistical Society: Series {B}\/}~{\em
  72\/}(3), 306--307.

\bibitem[\protect\citeauthoryear{Whiteley, Andrieu, and Doucet}{Whiteley
  et~al.}{2010}]{WhiteleyAD:2010}
Whiteley, N., C.~Andrieu, and A.~Doucet (2010).
\newblock Efficient {B}ayesian inference for switching state-space models using
  discrete particle {M}arkov chain {M}onte {C}arlo methods.
\newblock Technical report, Bristol Statistics Research Report 10:04.

\end{thebibliography}
